\documentclass{svjour3}                     
\usepackage[version=3]{mhchem} 
\usepackage{tabularx} 
\usepackage{algorithmic} 
\usepackage{algorithm} 
\usepackage{enumerate} 
\usepackage{appendix} 
\usepackage{amsmath}
\usepackage{amsfonts}
\usepackage{amsbsy}
\usepackage{amssymb}
\usepackage{amsthm}
\usepackage{graphicx,subfigure} 
\usepackage{caption}

\usepackage{appendix}

\usepackage{setspace}
\usepackage{array}
\usepackage{fmtcount} 
\usepackage{verbatim} 
\newtheorem{assumption}{Assumption}

\usepackage[left=45mm, right=45mm, top=30mm, bottom=30mm]{geometry}

\usepackage{tikz}
\usetikzlibrary{matrix,shapes,arrows,positioning,chains}

\tikzstyle{decision} = [diamond, draw, very thick, text width=4.5em, text badly centered, node distance=2.5cm]
\tikzstyle{block} = [rectangle, draw, very thick, text width=6em, text centered, rounded corners, node distance=2.0cm,  minimum height=2em]
\tikzstyle{line} = [draw, -latex']
\tikzstyle{cloud} = [draw, very thick, ellipse, node distance=2cm,minimum height=2em]

\usepackage{ctable} 
\begin{document}
	
	\title{A joint decomposition method for global optimization of multiscenario nonconvex mixed-integer nonlinear programs
	}

	\titlerunning{Joint decomposition for multiscenario MINLP}        
	
	\author{Emmanuel Ogbe        \and
		Xiang Li 
	}
	
	
	\institute{Emmanuel Ogbe \at
	Department of Chemical Engineering, 
		Queen's University, Kingston, ON, Canada, K7L 3N6\\         
		\and
		Xiang Li\at
		Department of Chemical Engineering, 
		Queen's University, Kingston, ON, Canada, K7L 3N6\\
		Tel.: +1-613-533-6582 \\
		Fax: +1-613-533-6637 \\
		\email{xiang.li@queensu.ca} \\         
	}
	
	\date{Received: date / Accepted: date}

 \maketitle

\begin{abstract}
	This paper proposes a joint decomposition method that combines Lagrangian decomposition and generalized Benders decomposition, to efficiently solve multiscenario nonconvex mixed-integer nonlinear programming (MINLP) problems to global optimality, without the need for explicit branch and bound search. In this approach, we view the variables coupling the scenario dependent variables and those causing nonconvexity as complicating variables. We systematically solve the Lagrangian decomposition subproblems and the generalized Benders decomposition subproblems in a unified framework. The method requires the solution of a difficult relaxed master problem, but the problem is only solved when necessary. Enhancements to the method are made to reduce the number of the relaxed master problems to be solved and ease the solution of each relaxed master problem. We consider two scenario-based, two-stage stochastic nonconvex MINLP problems that arise from integrated design and operation of process networks in the case study, and we show that the proposed method can solve the two problems significantly faster than state-of-the-art global optimization solvers. 
\end{abstract}


\newpage

\section{Introduction} \label{sec:sec1}

Global optimization is a field of mathematical programming devoted to obtaining global optimal solutions; and it has over the years found enormous applications in Process Systems Engineering (PSE). Mixed-integer nonlinear programs are global optimization problems where some decision variables are integer while others are continuous. Discrete decisions and nonconvex nonlinearities introduce combinatorial behavior for such problems \cite{floudas1995nam} \cite{tawarmalani2004goo}. 
Various applications of mixed-integer nonlinear programming for PSE systems include natural gas network design and operation \cite{li2010spp}, gasoline blending and scheduling problems \cite{li2011}, expansion of chemical processes \cite{sahinidis1991cpo}, reliable design of software \cite{Berman1993} \cite{floudas1999hot}, pump network problem \cite{adjiman2000goo} \cite{westerlund1994oop}, chemical process design synthesis \cite{Duran1986}, planning of facility investments for electric power generation \cite{Bloom1983}, etc.

As adopted for mixed-integer linear programing (MILP), branch-and-bound has been employed for global optimization of nonconvex mixed-integer nonlinear programs (MINLP) \cite{Falksoland1969} \cite{Soland1971} \cite{tawarmalani2004goo}. The method entails systematically generating lower and upper bounds of the optimal objective function value over subdomains of the search space. The lower bounds can be generated via convex relaxations (such as McCommick relaxations \cite{mccormick1976cog}) or \textit{Lagrangian relaxation} (or called \textit{Lagrangian decomposition}) \cite{Fisher1981}\cite{cynthia1998} \cite{Guignard1987}. Ways of generating multipliers for the Lagrangian subproblem exist, including subgradient methods \cite{Held1974}, cutting plane methods \cite{Fisher1981}, and the Dantzig-Wolfe master problem (also known the restricted Lagrangian master problem) \cite{van1983cross} \cite{OgbeLi2016}. 

Branch-and-bound based strategies can be improved by incorporation of domain reduction techniques. Domain reduction entails eliminating portions of the feasible domain based on feasibility and optimality. Bound tightening or contraction \cite{Zamora1999}, range reduction \cite{ryoo1996bat} and generation of cutting planes \cite{misener-floudas2014} are different domain reduction strategies that have been successful in solving nonconvex problems \cite{floudas1999hot}. In bound contraction, the variable bounds are shrunk at every iteration by solving bound contraction subproblems \cite{Zamora1999}. In range reduction, the bounds on the variables are shrunk based on simple calculations using Lagrange multiplier information \cite{ryoo1996bat}. For cutting planes generation, Lagrangian relaxation information provides \textit{cuts} that is used to cut-off portion of the feasible domain that does not contain the global optimum \cite{Karuppiah2008}. Current state-of-the-art commercial deterministic global optimization solvers embody branch-and-bound and enhancements such as tighter convex relaxations and domain reduction techniques, such as the Branch-And-Reduce Optimization Navigator (BARON) \cite{tawarmalani2004goo} and Algorithms for coNTinuous/Integer Global Optimization of Nonlinear Equations (ANTIGONE) \cite{misener-floudas:ANTIGONE:2014}. They do provide rigorous frameworks for global optimization of Problem \eqref{eq:P0}. 

Branch-and-bound based methods have been successful for global optimization, mostly for small to medium sized problems. However, when the size of the problem becomes large, the branch-and-bound steps needed for convergence can be prohibitively large. A typical example of large-scale nonconvex MINLP is the following multiscenario optimization problem: 
\begin{equation} \label{eq:P0} \tag{P0}
\begin{split} 
& \min_{\substack{x_0 \\ z_1,...,z_s}} \;  \sum_{\omega=1}^{s}[ f_{0,\omega}(x_0)+ f_{\omega}(v_{\omega})]\\ 
& \textrm{s.t.}\;\; g_{0,\omega}(x_0) + g_{\omega}(v_{\omega}) \le 0, \quad \forall {\omega} \in \{1,...,s\},\\   
& \quad \;\;\; {v_{\omega}} \in {V_{\omega}}, \quad \forall {\omega} \in \{1,...,s\}, \\ 
& \quad \;\;\; x_0 \in X_0, \\ 
\end{split}
\end{equation}
where $x_0$ links $s$ subparts of the model that are indexed by $\omega$, and it is called {\it linking variable} in the paper. We assume that at least one of the functions $f_{0,\omega}: X_0 \rightarrow \mathbb{R}$, $f_{\omega}: V_{\omega} \rightarrow \mathbb{R}$, $g_{0,\omega}: X_0 \rightarrow \mathbb{R}^m$ , $g_{\omega}: V_{\omega} \rightarrow \mathbb{R}^m$ or one of the sets $X_0$ and $V_{\omega}$ is nonconvex, so Problem \eqref{eq:P0} is a nonconvex MINLP, or a nonconvex nonlinear program (NLP) if no integer variables are involved. Clearly, \eqref{eq:P0} is a large-scale problem when $s$ is large. Problem \eqref{eq:P0} has attracted more and more attention over the last 20 years in the field of PSE \cite{sahinidis2004ouu}. It usually arises from scenario-based two-stage stochastic programming \cite{birge2010its} \cite{vanslyke1969}, for which $x_0$ represents the first stage decisions that are made before the uncertainty is realized and $v_\omega$ represents second-stage decisions that are made after the uncertainty is revealed in scenario $\omega$. Functions $f_{0,\omega}$ and $f_\omega$ represent probability times costs associated with $x_0$ and $v_{\omega}$ for every scenario $\omega$. Problem \eqref{eq:P0} can also arise from integrated system design and operation problems which consider system operation over multiple time periods (but without uncertainties), such as for energy polygeneration plants \cite{chen2011} and electrical power distribution networks \cite{frank2015710}). In this case, $x_0$ represents system design decisions and $x_\omega$ represents system operational decisions for time period (or scenario) $\omega$, and $f_{0,\omega}$ and $f_{\omega}$ represent frequency of occurrence of time period $\omega$ times investment cost and operational cost, respectively. In this paper, we focus on how to efficiently solve Problem (P0) to global optimality, rather than how to generate scenarios and probabilities for stochastic programming or the time periods and their occurrence frequencies for multiperiod optimization.   

It is well-known that Problem \eqref{eq:P0} has a decomposable  structure that could be exploited for efficient solution. 
\textit{Benders decomposition} (BD) \cite{benders1962pps} (known as L-shaped method in the stochastic programming literature \cite{vanslyke1969} \cite{birge2010its}) is one class of decomposition methods applied for MILPs. Geoffrion \cite{geoffrion1972gbd} generalized BD into \textit{Generalized Benders Decomposition} (GBD), for solving convex MINLPs. Li et al. developed a further extension, called \textit{Nonconvex Generalized Benders Decomposition} \cite{li2011}, for solving nonconvex MINLPs, but this method can guarantee global optimality only if the linking variable is fully integer. Karuppiah and Grossmann applied a Lagrangian decomposition-based scheme to solve Problem \eqref{eq:P0} \cite{Karuppiah2008lbb}; in order to guarantee convergence to a global optimum, explicit branch-and-bound of linking variables are needed. They also presented bound contraction as an optional scheme in their Lagrangian-based branch-and-bound strategy. Shim et al. \cite{Shim2013} proposed a method that combines Lagrangian decomposition and BD together with branch-and-bound (to ensure convergence), in order to solve a class of bilevel programs with an integer program in the upper-level and a complementarity problem in the lower-level. A more recent algorithm combining NGBD and Lagrangian decomposition was proposed by Kannan and Barton \cite{Rohit2015}, and this algorithm also requires explicit branch-and-bound for convergence.

Efforts have been taken to achieve better computational efficiency by combining classical decomposition methods. In 1983, Van Roy proposed a \textit{cross decomposition} method that combines Lagrangian decomposition and Benders decomposition \cite{van1983cross} to solve MILP problems which do not have non-linking integer variables. Since then, a number of extensions and variants of cross decomposition have been developed \cite{vanroy1984} \cite{Holmberg1990} \cite{Holmberg1997} \cite{Deep1993} \cite{Mitra2016} \cite{OgbeLi2016}. All of these methods require that no nonconvexity comes from non-linking variables as otherwise finite convergence cannot be guaranteed. 

The performance of branch-and-bound based solution methods depends heavily on the branching and node selection strategies, but what are the best strategies for a particular problem are usually unknown. In addition, branching and node selection strategies are not able to fully exploit the problem structure. Therefore, the goal of this paper is to develop a new decomposition method for global optimization of Problem \eqref{eq:P0}, which does not require explicit branch-and-bound. The new decomposition method was inspired by cross decomposition, and it follows a similar algorithm design philosophy, combining primarily generalized Benders decomposition and Lagrangian decomposition. However, its decomposition procedure is rather different in many details due to the nonconvexity it has to deal with, so we do not call it cross decomposition, but a new name {\it joint decomposition}. To the best of our knowledge, this is the first decomposition method that can solve Problem \eqref{eq:P0} to global optimality without explicitly performing branch-and-bound (but the solution of nonconvex subproblems requires branch-and-bound based solvers). 

The remaining part of the article is organized as follows. In section 2, we give a brief introduction to generalized Benders decomposition and Lagrangian decomposition, using a reformulation of Problem \eqref{eq:P0}. Then in section 3, we present the basic joint decomposition algorithm and the convergence proof. Section 4 discusses enhancements to the basic joint decomposition algorithm, including domain reduction and use of extra convex relaxation subproblems. The joint decomposition methods are tested with two case study problems adapted from the literature, and the simulation results demonstrate the effectiveness and the computational advantages of the methods. The article ends with concluding remarks in section 6. 

\section{Problem reformulation and classical decomposition methods} \label{sec:sec2}
In order to bring up the joint decomposition idea, we reformulate Problem \eqref{eq:P0} and briefly discuss how the reformulated problem can be solved via classical GBD and LD methods. The reformulation starts from separating the convex part and the nonconvex part of the problem, and it ends up in the following  form:

\begin{equation} \label{eq:P} \tag{P}
\begin{split} 
& \min_{\substack{x_0, x_{1},...,x_{s} \\ y_{1},...,y_{s}}} \;\sum_{\omega=1}^s c_\omega^Tx_{\omega}\\ 
& \textrm{s.t.}  \quad  \; \; x_{0}  = H_{\omega}x_{\omega}, \quad \forall \omega \in \{1,...,s\}, \\   
& \quad \quad \; \;\; A_{\omega}x_{\omega} + B_{\omega}y_{\omega} \le 0, \quad \forall \omega \in \{1,...,s\}, \\   
& \quad \quad \;\;\; x_{0} \in {X_0}, \\
& \quad \quad \;\;\; x_{\omega} \in X_{\omega}, \; {y_{\omega}} \in {Y_{\omega}}, \quad \forall \omega \in \{1,...,s\}, \\ 
\end{split}
\end{equation}
where set $X_\omega \subset \mathbb{R}^{n_x}$ is convex, set $Y_{\omega} \subset \mathbb{R}^{n_y}$ is nonconvex, and set $x_0 \subset \mathbb{R}^{n_0}$ can be either convex or nonconvex. The first group of equations in \eqref{eq:P} are \textit{nonanticipativity constraints} (NACs) \cite{Guignard1987}\cite{Caroe199937} \cite{Karuppiah2008}, where matrix $H_\omega \in \mathbb{R}^{n_0} \times \mathbb{R}^{n_{x}}$ selects from $x_{\omega}$ the duplicated $x_0$ for scenario $\omega$. The details of transforming \eqref{eq:P0} to \eqref{eq:P} are provided in Appendix A. 

$x_0$ and $y_\omega$ are the two reasons why Problem \eqref{eq:P} is difficult to solve. Linking variables $x_0$ couple different subparts of the model and they cause nonconvexity if set $X_0$ is nonconvex. Variables $y_{\omega}$ cause nonconvexity due to the nonconvexity of set $Y_\omega$. If the values of $x_0$ and $y_{\omega}$ are fixed, the problem will be much easier to solve. Therefore, in this paper we call $x_0$ and $y_{\omega}$ {\it complicating variables}. In order to distinguish the two sets of variables, we also call $x_0$ {\it linking variables}, and $y_0$ {\it non-linking complicating variables}. We also call $x_\omega$ {\it non-complicating variables}. 

The classical GBD method can be used to solve Problem (\ref{eq:P}) by treating $x_0$ and $y_{\omega}$ as complicating variables, while the LD method can be used to solve Problem (\ref{eq:P}) by dualizing NACs so that $x_0$ no long links different scenarios. In the next two subsections we briefly introduce GBD and LD for Problem (\ref{eq:P}), and we make the following assumptions for Problem (\ref{eq:P}) for convenience of discussion. 

\begin{assumption} \label{ass:boundedness}
$X_0$, $X_{\omega}$ and $Y_{\omega}$ for all $\omega  \in \{1,...,s\}$ are non-empty and compact. 
\end{assumption}

\begin{assumption} \label{ass:slater}
After fixing $(x_0, y_1, \cdots, y_s )$ to any point in $X_0 \times Y_1 \times \cdots \times Y_s$, if Problem (\ref{eq:P}) is feasible, it satisfies Slater condition.  
\end{assumption}

Assumption \ref{ass:boundedness} is a mild assumption, as for most real-world applications, the variables are naturally bounded and the functions involved are continuous. If a discontinuous function is involved, it can usually be expressed with continuous functions and extra integer variables. Assumption \ref{ass:slater} ensures strong duality of convex subproblems that is required for GBD. If this assumption is not satisfied for a problem, we can treat the non-complicating variables that fail the Slater condition to be complicating variables, so that after fixing all complicating variables the Slater condition is satisfied.

\subsection{Generalized Benders decomposition}
At each GBD iteration $l$, fixing the complicating variables $x_0=x_0^{(l)}$, $y_{\omega}=y_{\omega}^{(l)}$ ($\forall \omega \in \{1,...,s\}$) results in an upper bounding problem that can be decomposed into the following Benders primal subproblem for each scenario $\omega$:
\begin{equation} \label{eq:BPP} \tag{BPP$_{\omega}^{(l)}$}
\begin{split}
obj_{\textrm{BPP}_\omega^{(l)}} = & \min_{\substack{ x_{\omega}}} \; \; \; c_\omega^Tx_{\omega}\\ 
& \textrm{s.t.}  \quad  \; \; x_{0}^{(l)}  = H_{\omega}x_{\omega}, \quad \\   
& \quad \quad \; \;A_{\omega}x_{\omega} + B_{\omega}y_{\omega}^{(l)} \le 0, \quad \\   
& \quad \quad \;\;  x_{\omega} \in X_{\omega}, \\ 
\end{split}
\end{equation}
$obj_{\textrm{BPP}_\omega^l}$ is the optimal objective value of \eqref{eq:BPP}. For convenience, we indicate the optimal objective value of a problem in the above way for all subproblems discussed in this paper. Obviously, $\sum_{\omega=1}^{s}obj_{\textrm{BPP}_\omega^{(l)}}$ represents an upper bound for Problem \eqref{eq:P}. If \eqref{eq:BPP} is infeasible for one scenario, then solve the following Benders feasibility subproblem for each scenario $\omega$:
\begin{equation} \label{eq:BFP} \tag{BFP$_{\omega}^{(l)}$}
\begin{split}
obj_{\textrm{BFP}_\omega^{(l)}}= & \min_{\substack{ x_{\omega},z_{1,\omega}^+,z_{1,\omega}^-,z_{2,\omega}}} \; \; \; ||{z}_{1,\omega}^+|| + ||{z}_{1,\omega}^-|| + ||z_{2,\omega}||\\  
& \textrm{s.t.}  \quad  x_{0}^{(l)}  = H_{\omega}x_{\omega}+ z_{1,\omega}^+ - z_{1,\omega}^-, \\   
& \quad \quad \; \;A_{\omega}x_{\omega} + B_{\omega}y_{\omega}^{(l)} \le {z}_{2,\omega},  \\   
& \quad \quad \;\;  x_{\omega} \in X_{\omega}, \quad z_{1,\omega}^+,z_{1,\omega}^-,z_{2,\omega} \ge 0, \\ 
\end{split}
\end{equation}
where ${z}_{1,\omega}^+$, ${z}_{1,\omega}^-$, and ${z}_{2,\omega}$ are slack variables. Note that \eqref{eq:BFP} is always feasible according to Assumption \ref{ass:boundedness}. Solution of \eqref{eq:BFP} provides a feasibility cut (that is described below), which prevents the generation of the same infeasible $x_0^{l}$ and $y_\omega^{(l)}$ \cite{lasdon1970spp}. 

At the same iteration, the following Benders relaxed master problem is solved to yield a lower bound for Problem \eqref{eq:P}:
\begin{equation} \label{eq:BRMP}  \tag{BRMP$^{(l)}$}
\begin{split}
&\min_{\substack{x_0,\eta_0,\eta_1,...,\eta_s \\ y_{1},...,y_{s}}} \eta_0 \\ 
& \textrm{s.t.} \quad \eta_0 \ge \sum_{\omega=1}^{s}\eta_\omega \quad \\ 
& \quad \quad \eta_\omega  \ge  obj_{\textrm{BPP}_\omega^{(j)}} + ({ \lambda_{\omega}}^{(j)})^{\textrm{T}} B_{\omega}(y_{\omega}-y_{\omega}^{(j)})+  ({ \mu_{\omega}^{(j)}})^{\textrm{T}}\left( x_0 - x_0^{(j)} \right), \\
& \qquad \qquad \qquad \qquad \qquad \qquad \qquad \qquad  \forall \omega \in \{1,...,s\}, \quad \forall  j \in {T^{(l)}}, \\ 
& \quad \quad 0 \ge  obj_{\textrm{BFP}_\omega^{(j)}} + ({ \lambda_{\omega}}^{(j)})^{\textrm{T}} B_{\omega}(y_{\omega}-y_{\omega}^{(j)})+  ({ \mu_{\omega}^{(j)}})^{\textrm{T}}\left( x_0 - x_0^{(j)} \right), \\
& \qquad \qquad \qquad \qquad \qquad \qquad \qquad \qquad \forall \omega \in \{1,...,s\}, \quad \forall  j \in {S^{(l)}}, \\  
& \quad \quad x_0\in X_0,\\
& \quad \quad y_{\omega} \in Y_{\omega}, \quad \forall \omega \in \{1,...,s\},  \\ 
\end{split}
\end{equation}
where $\mu_\omega^{(l)}$ includes Lagrange multipliers for the first group of constraints in Problem \eqref{eq:BPP} or \eqref{eq:BFP}, and $\lambda_\omega^{(l)}$ includes Lagrange multipliers for the second group of constraints in Problem \eqref{eq:BPP} or \eqref{eq:BFP}. Set $T^{(l)}$ includes indices of Benders iterations at which only \eqref{eq:BPP} is solved, and set $S^{(l)}$ includes indices of Benders iterations at which \eqref{eq:BFP} is solved. Note that Problem \eqref{eq:BRMP}) is used in the multicut BD or GBD, which is different from the one used in the classical single cut BD or GBD. The multicut version of the Benders master problem is known to be tighter than the single cut version \cite{birge1988mat} \cite{OgbeLi2015b}, so it is considered in this paper.

\begin{remark} \label{rem:rem1}
The finite convergence property of GBD is stated and proved in \cite{geoffrion1972gbd}. In Section 3, we will provide more details in the context of our new decomposition method.  
\end{remark}

\begin{remark} \label{rem:rem2}
For \eqref{eq:P}, the relaxed master problem \eqref{eq:BRMP} can still be very difficult as its size grows with the number of scenarios. However, if most variables in \eqref{eq:P} are non-complicating variables, the size of \eqref{eq:BRMP} is much smaller than that of \eqref{eq:P}, and then \eqref{eq:BRMP} is much easier to solve than \eqref{eq:P}.   
\end{remark}

\subsection{Lagrangian decomposition}
We start discussing LD from the Lagrangian dual of Problem \eqref{eq:P} that is constructed by dualizing the NACs of the problem:
\begin{equation} \label{eq:DP}  \tag{DP}
obj_{\textrm{DP}}=\max_{\pi_1, \cdots,\pi_s \ge 0} obj_{\textrm{LS}}(\pi_1, \cdots,\pi_s),
\end{equation}
where $obj_{\textrm{LS}}(\pi_1, \cdots,\pi_s)$ is the optimal objective value of the following Lagrangian subproblem with given $(\pi_1, \cdots,\pi_s)$: 
\begin{equation} \label{eq:LS}  \tag{LS$(\pi_1, \cdots,\pi_s)$}
\begin{split} 
& \min_{\substack{x_0, x_{1},...,x_{s} \\ y_{1},...,y_{s}}} \;\sum_{\omega=1}^s[ c_\omega^Tx_{\omega}+\pi_\omega^T(x_0-H_{\omega}x_{\omega})]\\ 
& \textrm{s.t.}  \quad  \;\; A_{\omega}x_{\omega} + B_{\omega}y_{\omega} \le 0, \quad \forall \omega \in \{1,...,s\}, \\   
& \quad \quad \;\;\; x_{0} \in {X}_0, \\
& \quad \quad \;\;\; x_{\omega} \in X_{\omega}, {y_{\omega}} \in {Y_{\omega}}, \quad \forall \omega \in \{1,...,s\}. \\ 
\end{split}
\end{equation}
Due to weak duality, Problem \eqref{eq:DP} or any Lagrangian subproblem is a lower bounding problem for Problem \eqref{eq:P}. Typically, the LD method is incorporated in a branch-and-bound framework that only needs to branch on linking variables $x_0$ to guarantee convergence to an $\epsilon$-optimal solution. At each branch-and-bound node or LD iteration $k$, a set of multipliers $(\pi_1^k, \cdots, \pi_s^k)$ are selected to construct a Lagrangian subproblem for \eqref{eq:DP}, and this subproblem can be naturally decomposed into $s+1$ subproblems, i.e.,
\begin{equation} \label{eq:LS0}   \tag{LS$_0^k$}
\begin{split} 
obj_{\textrm{LS}^k_0} = & \min_{x_{0}} \; \sum_{\omega=1}^s(\pi_\omega^k)^Tx_{0} \\  
&  \textrm{s.t}  \quad  \;\;\; {x_{0}} \in {X_{0}},
\end{split}
\end{equation}
and 
\begin{equation} \label{eq:LSw}  \tag{LS$_{\omega}^k$}
\begin{split} 
& \min_{\substack{ x_{\omega}, y_{\omega}}} \; c_\omega^Tx_{\omega}-(\pi_\omega^k)^T H_{\omega}x_{\omega}\\ 
& \textrm{s.t.}  \quad \; \;  A_{\omega}x_{\omega} + B_{\omega}y_{\omega} \le 0, \quad \\   
& \quad \quad \;\;\; x_{\omega} \in X_{\omega}, \quad  {y_{\omega}} \in {Y_{\omega}}, \quad  \\ 
\end{split}
\end{equation}
for all $\omega \in \{1,\cdots, s\}$. Let $obj_{LS^k}$ be the optimal objective value of the Lagrangian subproblem, then $obj_{\textrm{LS}^k}=\sum_{\omega=1}^{s}obj_{{LS}_\omega^k} +obj_{{LS_0}^k}$. Clearly, $obj_{\textrm{LS}^k} \leq obj_{\textrm{DP}}$ always holds. If $(\pi_1^k, \cdots, \pi_s^k)$ happens to be an optimal solution of \eqref{eq:DP}, then $obj_{\textrm{LS}^k}=obj_{\textrm{DP}}$.

The upper bounds in the LD methods are typically generated by fixing $x_0$ to certain values. At each iteration $k$, an upper bounding problem, or called primal problem, is constructed via fixing $x_0=x_0^k$ (which may be the solution of \eqref{eq:LS0}), and this problem can be separated into $s$ primal subproblem in the following form:
\begin{equation}  \label{eq:PP} \tag{PP$_{\omega}^k$}
\begin{split}
obj_{\textrm{PP}_\omega^k} = &\min_{\substack{ x_{\omega}, y_{\omega}}} \; \; \; c_\omega^Tx_{\omega}\\ 
& \textrm{s.t.}  \quad  \;\; x_{0}^k  = H_{\omega}x_{\omega}, \quad \\   
& \quad \quad \; \;\; A_{\omega}x_{\omega} + B_{\omega}y_{\omega} \le 0, \quad \\   
& \quad \quad \;\;\;  x_{\omega} \in X_{\omega}, \quad  {y_{\omega}} \in {Y_{\omega}}, \quad \\ 
\end{split}
\end{equation}
Let $obj_{{PP}^k}$ be the optimal objective value of the primal problem, then $obj_{{PP}^k}=\sum_{\omega=1}^sobj_{\textrm{PP}_\omega^k}$.

For generation of multipliers, we take the idea from Dantzig-Wolfe decomposition, which is essentially a special LD method. Consider the convex hull of nonconvex set $Y_\omega$:
\[\tilde{Y}_{\omega}=\{y_{\omega} \in \mathbb{R}^{n_{y}}:y_{\omega}=\sum_{i \in I}^{}\theta_\omega^{[i]}y_{\omega}^{[i]}, \; \sum_{i \in I} \theta_\omega^{[i]}=1, \; \theta_\omega^{[i]} \ge 0, \forall i \in I\}, \]
where $y_{\omega}^{[i]}$ denotes a point in $Y_{\omega}$ that is indexed by $i$. The index set $I$ may need to be an infinite set for $\tilde{Y}_{\omega}$ being the convex hull. Replace $Y_\omega$ with its convex hull for all $\omega$ in \eqref{eq:P}, then we get the following Dantzig-wolfe master problem, or called primal master problem in this paper:
\begin{equation}  \label{eq:PMP} \tag{PMP}
\begin{split}
& \min_{\substack{x_0, \theta_1^{[i]},...,\theta_s^{[i]}\\ x_{1},...,x_{s} }} \;\sum_{\omega=1}^sc_\omega^Tx_\omega \\ 
& \textrm{s.t.}  \quad \quad \; \;\; x_{0}  = H_{\omega}x_{\omega}, \quad \forall \omega \in \{1,...,s\}, \\   
& \quad \quad \quad \; \;\; A_{\omega}x_{\omega} + B_{\omega}\sum_{i \in I}\theta_\omega^{[i]}y_{\omega}^{[i]} \le 0, \quad \forall \omega \in \{1,...,s\}, \\   
& \quad \quad \quad \;\;\; \sum_{i \in I}\theta_\omega^{[i]} = 1, \quad  \theta_\omega^{[i]} 
\ge 0 ,  \quad \forall i \in I, \quad \forall \omega \in \{1,...,s\}, \\
& \quad \quad \quad \;\;\; x_0 \in X_0, \\
& \quad \quad \quad \;\;\;  x_{\omega} \in X_{\omega}, \quad \forall \omega \in \{1,...,s\}  \\ 
\end{split}
\end{equation}
Clearly, Problem \eqref{eq:PMP} is a relaxation of Problem \eqref{eq:P}, and it is either fully convex or partially convex (as set $X_0$ can still be nonconvex). At LD iteration $k$, the following restriction of \eqref{eq:PMP} can be solved: 


\begin{equation}  \label{eq:RPMP} \tag{RPMP$^k$}
\begin{split}
& \min_{\substack{x_0,\theta_1^{[i]},...,\theta_s^{[i]} \\ x_{1},...,x_{s} }} \;\sum_{\omega=1}^s c_\omega^Tx_\omega\\ 
& \textrm{s.t.}  \quad \quad  \; \;\; x_{0}  = H_{\omega}x_{\omega}, \quad \forall \omega \in \{1,...,s\}, \\   
& \quad \quad  \quad \; \;\; A_{\omega}x_{\omega} + B_{\omega}\sum_{i \in I^k}\theta_\omega^{[i]}y_{\omega}^{[i]} \le 0, \quad \forall \omega \in \{1,...,s\}, \\   
& \quad \quad \quad \;\;\; \sum_{i \in I^k}\theta_\omega^{[i]} = 1, \quad \theta_\omega^{[i]} 
\ge 0 , \quad \forall i \in I^k,\quad \forall \omega \in \{1,...,s\},\\
& \quad \quad \quad \;\;\; x_0 \in X_0, \\
& \quad \quad \quad \;\;\; x_{\omega} \in X_{\omega}, \quad \forall \omega \in \{1,...,s\}, 
\end{split}
\end{equation}
where index set ${I^{k}} \subset I$ is finite. ${I^{k}}$ may consist of indices of $y_\omega$ that are generated in the previously solved primal problems and Lagrangian subproblems. Replacing set $I$ with set $I^k$ is a restriction operation, so \eqref{eq:RPMP} is a restriction of \eqref{eq:PMP}. Since \eqref{eq:PMP} is a relaxation of \eqref{eq:P}, \eqref{eq:RPMP} is neither a relaxation nor a restriction of \eqref{eq:P}, so it does not yield an upper or a lower bound of \eqref{eq:P}. The role of \eqref{eq:RPMP} in joint decomposition is to generate multipliers for NACs in order to construct a Lagrangian subproblem for iteration $k$. Problem \eqref{eq:RPMP} can be solved by a state-of-the-art optimization solver directly or by GBD.

Actually, we can construct a different Lagrangian dual of Problem \eqref{eq:P} by dualizing both the NACs and the second group of constraints in the problem, as what we do for GBD in the last subsection. However, this Lagrangian dual is not as tight as Problem \eqref{eq:DP} (as stated by the following proposition), so it is not preferred for a LD method. The following proposition follows from Theorem 3.1 of \cite{Guignard1987} and its proof is omitted here.

\begin{proposition} \label{prop:dualgap}
Consider the following Lagrangian dual of Problem \eqref{eq:P}:
\begin{equation} \label{eq:DP2}  \tag{DP2}
obj_{\emph{DP2}}=\max_{\substack{\mu_1,\cdots,\mu_s \ge 0 \\ \lambda_1, \cdots, \lambda_s \ge 0}} obj_{\emph{LS2}} (\mu_1, \cdots, \mu_s, \lambda_1, \cdots, \lambda_s),
\end{equation}
where 
\begin{equation*} \label{eq:LS2}  
\begin{split} 
obj_{\emph{LS2}}=& \min_{\substack{x_0, x_{1},...,x_{s} \\ y_{1},...,y_{s}}} \;\sum_{\omega=1}^s[ c_\omega^Tx_{\omega}+\mu_\omega^T(x_0-H_{\omega}x_{\omega}) + \lambda_\omega^T (A_{\omega}x_{\omega} + B_{\omega}y_{\omega}) ]\\ 
& \textrm{s.t.}  \quad  \;\; x_{0} \in {X}_0, \\
& \quad \quad \;\;\; x_{\omega} \in X_{\omega}, {y_{\omega}} \in {Y_{\omega}}, \quad \forall \omega \in \{1,...,s\}. \\ 
\end{split}
\end{equation*}
The dual gap of \eqref{eq:DP} is no larger than the dual gap of \eqref{eq:DP2}. 
\end{proposition}

\section{The joint decomposition method} \label{sec:sec3}

\subsection{Synergizing LD and GBD}

In the LD method described in the last section, at each iteration the subproblems to be solved are much easier than the original problem \eqref{eq:P}, as either the size of the subproblem is independent of number of scenarios, such as \eqref{eq:PP}, \eqref{eq:LS0}, and \eqref{eq:LSw}, or the subproblem is a MILP or convex MINLP that can be solved by existing optimization solvers or by GBD relatively easily, such as \eqref{eq:RPMP}. However, without branching on the linking variables $x_0$, LD cannot guarantee finding a global solution, and we do not always know how to exploit the problem structure to efficiently branch on $x_0$ and whether the branching can be efficient enough. 

On the other hand, GBD can find a global solution, but it requires solving the nonconvex relaxed master problem \eqref{eq:BRMP} at each iteration. The size of \eqref{eq:BRMP} may be much smaller than the size of \eqref{eq:P} if most variables in \eqref{eq:P} are non-complicating variables, but \eqref{eq:BRMP} can still be difficult to solve, especially considering that it needs to be solved at each iteration and its size grows with the number of iterations. 

Therefore, there may be a way to combine LD and GBD, such that we solve as many as possible LD subproblems and Benders primal subproblems \eqref{eq:BPP} (as they are relatively easy to solve), but avoid solving many difficult Benders relaxed master problems \eqref{eq:BRMP}. This idea is similar to the one that motivates cross decomposition \cite{van1983cross}, but it leads to very different subproblems and a very different algorithmic procedure. The subproblems are very different, because for problem \eqref{eq:P}, we prefer dualizing only NACs in LD in order to achieve the smallest possible dual gap (according to Proposition \ref{prop:dualgap}), but we have to dualize both the NACs and the second group of constraints in GBD. In addition, due to the different nature of the subproblems, the order in which the subproblems are solved and how often the problems are solved are different. Therefore, we do not name the proposed method cross decomposition, but call it {\it joint decomposition} (JD).

Fig. \ref{fig:figure1} shows the basic framework of JD. Each JD iteration includes one LD iteration part, as indicated by the solid lines, and possibly one GBD iteration, as indicated by the dashed lines. In a JD iteration, the GBD iteration is performed only when the LD iteration improves over the previous LD iteration substantially. The GBD iteration is same to the one described in the last section, except that the relaxed master problem \eqref{eq:BRMP} includes more valid cuts (which will be described later). The LD iteration is slightly different from the one described in the last section. One difference is that, after solving \eqref{eq:PP} at LD iteration $k$, a Benders primal problem (BPP$^k$) is constructed using $x_0^k$ (which is used for constructing \eqref{eq:PP}) and $(y_{1}, \cdots, y_{s})$ (which is from the optimal solution of \eqref{eq:PP}). The (BPP$^k$) is solved to generate a Benders cut that can be added to \eqref{eq:BRMP}. The other difference is that \eqref{eq:RPMP}, \eqref{eq:LS0}, \eqref{eq:LSw} (decomposed from (LS$^k$)) slightly differ from the ones described in the last section, and they will be described later.     

\begin{remark} \label{rem:rem3}
The JD method requires that all subproblems can be solved using an existing optimization solver within reasonable time. If this requirement is not met, then JD does not work, or we have to further decompose the difficult subproblems into smaller, solvable subproblems.  
\end{remark}

\usetikzlibrary{arrows,positioning,shapes}

\begin{figure}[tbp]
	\begin{center}
		\centering
		\begin{tikzpicture}[node distance=10mm, >=latex',
		block/.style = {draw=black,rounded corners,very thick, rectangle, minimum height=10mm, minimum width=28mm,align=center},
		rblock/.style = {draw, rectangle, very thick, rounded corners=0.5em},
		tblock/.style = {draw, trapezium, minimum height=10mm, 
			trapezium left angle=75, trapezium right angle=105, align=center},
		]
		\node [rblock]                      (start)     {Initialize};
		\node [block, below=of start]       (Primal)   {PP$^k$};
		\node [block, right=of Primal]     (Bendersprimal)  {BPP$^{(l)}$};
		\node [block, below=of Bendersprimal]    (Relaxed)      {RMP$^{(l)}$};
		\node [block, below=of Primal]     (Lagrange)  {LS$^k$};
		\node [block, left=of Primal]    (Bendersprimal2)   {BPP$^k$};
		\node [block, below=of Bendersprimal2]     (Restricted) {RPMP$^{k}$};
		\path[draw,->] (start)      edge (Primal)
		(Primal)    edge (Bendersprimal2)
		(Lagrange)    edge (Primal)    
		(Bendersprimal2)    edge (Restricted)
		(Restricted)  edge (Lagrange)
		;
		\path[draw,->,dashed]
		    (Bendersprimal)   edge (Primal)
			(Relaxed)       edge  (Bendersprimal)
			(Lagrange) 	edge  (Relaxed) 
			;  
		  \node[anchor=north west,text width=12cm] (note2) at (-6,-5.0)
		  {
   		  RPMP$^{k}$: Restricted Primal Master Problem \\
          LS$^k$: Lagrangian subproblem, decomposed into \eqref{eq:LS0} and \eqref{eq:LSw} ($\omega=1,\cdots,s$).
		  RMP$^{(l)}$: Relaxed Master Problem, with extra cuts from LS$^k$ and BPP$^k$.  \\
		  BPP$^{(l)}$: Benders Primal Problem, decomposed into \eqref{eq:BPP} ($\omega=1,\cdots,s$).  \\
		  PP$^k$: Primal Problem, decomposed into \eqref{eq:PP} ($\omega=1,\cdots,s$). \\
          BPP$^k$: Benders Primal Problem, solved after PP$^k$ is solved. \\
		  
		  	}; 
		\end{tikzpicture}
				\caption{The basic joint decomposition framework}     
				 \label{fig:figure1}	
			\end{center}	
\end{figure}

\subsection{Feasibility issues}

According to Assumption \ref{ass:boundedness}, a subproblem in JD either has a solution or is infeasible. Here we explain how JD handles infeasibility of a subproblem. 

First, if a lower bounding problem (LS$^k$) or \eqref{eq:BRMP} is infeasible, then the original problem \eqref{eq:P} is infeasible and JD can terminate. 

Second, if (BPP$^k$) or (BPP$^{(l)}$) is infeasible, then JD will solve the corresponding Benders feasibility problem (BFP$^k$) or (BFP$^{(l)}$) to yield a feasibility cut. If (BFP$^k$) or (BFP$^{(l)}$) is infeasible, then \eqref{eq:P} is infeasible and JD can terminate.

Third, if \eqref{eq:PP} is infeasible, then JD will solve a feasibility problem that "softens" the second group of constraints:  and this problem can be separated into $s$ subproblems as follows:
\begin{equation}  \label{eq:FPw} \tag{FP$_{\omega}^k$}
\begin{split}
&\min_{\substack{ x_{\omega}, y_{\omega},{z}_\omega}} \; \; \; ||{z}_\omega||\\ 
& \textrm{s.t.}  \quad  \; \; x_{0}^k  = H_{\omega}x_{\omega}, \quad \\   
& \quad \quad \; \;\; A_{\omega}x_{\omega} + B_{\omega}y_{\omega} \le {z}_\omega, \quad \\   
& \quad \quad \;\;\;  x_{\omega} \in X_{\omega}, \quad {y_{\omega}} \in {Y_{\omega}}, \quad {z}_\omega \ge 0. 
\end{split}
\end{equation}
If \eqref{eq:FPw} is infeasible for one scenario $\omega$, then \eqref{eq:P} is infeasible and JD can terminate. If \eqref{eq:FPw} is feasible for all scenarios, then JD can construct and solve a feasible Benders feasibility problem (BFP$^k$) to yield a Benders feasibility cut for \eqref{eq:BRMP}. 

Finally, problem \eqref{eq:RPMP} can actually be infeasible if none of the $(y^{[i]}_1, \cdots, y^{[i]}_s)$ in the problem is feasible for the original problem \eqref{eq:P}. To prevent this infeasibility, we can generate a point $(\hat{y}_1, \cdots, \hat{y}_s)$ that is feasible for \eqref{eq:P}, by solving the following initial feasibility problem:
\begin{equation} \label{eq:IFP} \tag{IFP}
\begin{split} 
& \min_{\substack{x_0, x_{1},\cdots,x_{s} \\ y_{1},\cdots,y_{s} \\ z_{1}, \cdots, z_{\omega}}} \;\sum_{\omega=1}^s ||z_{\omega}||\\ 
& \textrm{s.t.}  \quad  \; \; x_{0}  = H_{\omega}x_{\omega}, \quad \forall \omega \in \{1,...,s\}, \\   
& \quad \quad \; \;\; A_{\omega}x_{\omega} + B_{\omega}y_{\omega} \le z_\omega, \quad \forall \omega \in \{1,...,s\}, \\   
& \quad \quad \;\;\; x_{0} \in {X_0}, \\
& \quad \quad \;\;\; x_{\omega} \in X_{\omega}, \; {y_{\omega}} \in {Y_{\omega}}, \; z_\omega \ge 0, \quad \forall \omega \in \{1,...,s\}. \\ 
\end{split}
\end{equation}
Problem \eqref{eq:IFP} is not naturally decomposable over the scenarios, but it can be solved by JD. When solving \eqref{eq:IFP} using JD, the restricted primal master problem \eqref{eq:RPMP} must have a solution (according to Assumption \ref{ass:boundedness}).

\subsection{The tightened subproblems} 

The relaxed master problem described in Section 2 can be tightened with the solutions of previously solved subproblems in JD. The tightened problem, called joint decomposition relaxed master problem, can be written as:
\begin{equation} \label{eq:JRMP}  \tag{JRMP$^{(l)}$}
\begin{split}
&\min_{\substack{x_0,\eta_0,\eta_1,...,\eta_s \\ y_{1},...,y_{s}}} \quad \eta_0 \\ 
& \textrm{s.t.} \;\;  \eta_0 \ge \sum_{\omega=1}^{s}\eta_\omega, \\ 
& \quad \quad  \eta_\omega  \ge  obj_{\textrm{BPP}_\omega^{(j)}} + ( \lambda_\omega^{(j)})^{\textrm{T}} B_{\omega}(y_{\omega}-y_{\omega}^{(j)})+  ( \mu_\omega^{(j)})^{\textrm{T}}\left( x_0 - x_0^{(j)} \right), \\
& \qquad \qquad \qquad \qquad \qquad \qquad \qquad \qquad  \forall \omega \in \{1,...,s\}, \quad \forall  j \in {T^{(l)}}, \\ 
& \quad \quad  0  \ge obj_{\textrm{BFP}_\omega^{(j)}} + (\lambda_\omega^{(j)})^{\textrm{T}} B_{\omega}(y_{\omega}-y_{\omega}^{(j)})+ ( \mu_\omega^{(j)})^{\textrm{T}}\left( x_0 - x_0^{(j)} \right), \\
& \qquad \qquad \qquad \qquad \qquad \qquad \qquad \qquad \forall \omega \in \{1,...,s\}, \quad  
 \forall  j \in {S^{(l)}}, \\ 
& \quad \quad \eta_\omega  \ge  obj_{\textrm{BPP}_\omega^j} + ( \lambda_\omega^j)^{\textrm{T}} B_{\omega}(y_{\omega}-y_{\omega}^j)+  ({ \mu_{\omega}^{j}})^{\textrm{T}}\left( x_0 - x_0^{j} \right), \\
& \qquad \qquad \qquad \qquad \qquad \qquad \qquad \qquad
\forall \omega \in \{1,...,s\}, \quad \forall  j \in {T^k}, \\ 
& \quad \quad  0  \ge obj_{\textrm{BFP}_\omega^j} + (\lambda_\omega^j)^{\textrm{T}} B_{\omega}(y_{\omega}-y_{\omega}^{j})+ ( \mu_\omega^j)^{\textrm{T}}\left( x_0 - x_0^{j} \right), \\
& \qquad \qquad \qquad \qquad \qquad \qquad \qquad \qquad \forall \omega \in \{1,...,s\}, \quad  \forall  j \in S^k, \\ 
& \quad \quad \eta_0 \le UBD, \\ 
& \quad \quad \eta_0 \ge LBD, \\    
& \quad \quad \eta_\omega  \ge  obj_{\textrm{LS}_\omega^i} +  ({ \pi_{\omega}^{i}})^{\textrm{T}}x_0, \quad \forall \omega \in \{1,...,s\}, \quad \forall  i \in {R^k}, \\ 
& \quad \quad x_0\in X_0,\quad y_{\omega} \in Y_{\omega}, \quad \forall \omega \in \{1,...,s\},  \\ 
\end{split}
\end{equation}
where the index set $R^k=\{1, \cdots, k \}$, $UBD$ is the current best upper bound for \eqref{eq:P}, and $LBD$ is the current best lower bound for \eqref{eq:P}. 

\begin{proposition} \label{prop:JRMP}
Problem \eqref{eq:JRMP} is a valid lower bounding problem for Problem \eqref{eq:P}.
\end{proposition}

\begin{proof}
Since it is already known that Problem \eqref{eq:BRMP} is a valid lower bounding problem and $UBD$ and $LBD$ are valid upper and lower bounds, we only need to prove that the cuts from Lagrangian subproblems together with the Benders optimality cuts do not exclude an optimal solution. Let $obj_\text{P}$ be the optimal objective value of \eqref{eq:P}, then
\[obj_{\text{P}} = \sum_{\omega=1}^{s} obj_{\text{PP}_{\omega}}(x_0),  \]
where
\[obj_{\text{PP}_{\omega}}(x_0)=\min\{c^T_{\omega}x_{\omega}: x_0=H_{\omega}x_{\omega}, \; A_{\omega} x_{\omega} + B_{\omega} y_{\omega} \leq 0, \; x_{\omega} \in X_{\omega}, \; y_{\omega} \in Y_{\omega}  \}.  \]

On the one hand, $\forall \pi^i_{\omega}, i \in R^k$, 
\begin{equation} \label{eq:LDoptcut}
\begin{split}
     & obj_{\text{PP}_{\omega}}(x_0) \\
\geq & \min\{c^T_{\omega}x_{\omega} + (\pi^i_{\omega})^T (x_0-H_{\omega}x_{\omega}): A_{\omega} x_{\omega} + B_{\omega} y_{\omega} \leq z_{\omega}, \; x_{\omega} \in X_{\omega}, \; y_{\omega} \in Y_{\omega}  \} \\
= & obj_{\textrm{LS}_\omega^i} +  ({ \pi_{\omega}^{i}})^{\textrm{T}}x_0. 
\end{split}
\end{equation}

On the other hand, 
\[obj_{\text{PP}_{\omega}}(x_0) = \min_{y_{\omega} \in Y_{\omega}} v_{\omega} (x_0, y_{\omega}),\] 
where $v_{\omega} (x_0, y_{\omega})=\min\{c^T_{\omega}x_{\omega}:x_0=H_{\omega}x_{\omega}, \; A_{\omega} x_{\omega} + B_{\omega} y_{\omega} \leq 0 \}$. From weak duality, $\forall j \in T^{(l)}$,  
\begin{equation*} 
\begin{split}
     & v_{\omega}(x_0,y_{\omega}) \\
\geq & \min\{c^T_{\omega}x_{\omega} + (\lambda_\omega^{(j)})^{\textrm{T}} (A_{\omega}x_{\omega} + B_{\omega}y_{\omega})+  ( \mu_\omega^{(j)})^{\textrm{T}} (x_0 - H_{\omega}x_{\omega} ): x_{\omega} \in X_{\omega} \} \\
= & obj_{\textrm{BPP}_\omega^{(j)}} + ( \lambda_\omega^{(j)})^{\textrm{T}} B_{\omega}(y_{\omega}-y_{\omega}^{(j)})+  ( \mu_\omega^{(j)})^{\textrm{T}}\left( x_0 - x_0^{(j)} \right).
\end{split}
\end{equation*}
Thefore, $\forall y_{\omega} \in Y_{\omega}$,
\begin{equation} \label{eq:BDoptcut}
obj_{\text{PP}_{\omega}}(x_0) \geq obj_{\textrm{BPP}_\omega^{(j)}} + ( \lambda_\omega^{(j)})^{\textrm{T}} B_{\omega}(y_{\omega}-y_{\omega}^{(j)})+  ( \mu_\omega^{(j)})^{\textrm{T}}\left( x_0 - x_0^{(j)} \right).
\end{equation}

Equations \eqref{eq:LDoptcut}-\eqref{eq:BDoptcut} indicate that the cuts from Lagrangian subproblems together with the Benders optimality cuts do not exclude an optimal solution of \eqref{eq:P}. 

\end{proof}

For convenience, we call the cuts from the Lagrangian subproblems, Lagrangian cuts. The Benders cuts and the Lagrangian cuts in \eqref{eq:JRMP} imply that, $\forall i \in R^k$, 
\[UBD \ge \eta_0 \ge \sum_{\omega=1}^s \eta_\omega \ge \sum_{\omega=1}^sobj_{LS_\omega^i} + \sum_{\omega=1}^s { { ({ \pi_{\omega}^{i}})^{\textrm{T}}} } x_0.
 \]
Now we get new constraints
\begin{equation*} \label{eq:star} \tag{*}
UBD \ge  \sum_{\omega=1}^s obj_{LS_\omega^i} + \sum_{\omega=1}^s  (\pi_\omega^i)^{\textrm{T}} x_0, \quad \forall i \in R^k, 
\end{equation*}
which only include variable $x_0$ and do not link different scenarios. This constraint can be used to enhance any subproblems that involves $x_0$ as variables. Specifically, problems \eqref{eq:LS0}, \eqref{eq:LSw}, \eqref{eq:RPMP} can be enhanced as: 
\begin{equation} \label{eq:LSw} \tag{LS$_\omega^k$}  
\begin{split} 
& \min_{\substack{ x_{c,\omega}, y_{nc,\omega}}} \; c_\omega^Tx_{\omega}-(\pi_\omega^k)^T H_{\omega}x_{\omega} \\ 
& \textrm{s.t.}  \quad \;\; A_{\omega}x_{\omega} + B_{\omega}y_{\omega} \le 0, \quad \\  
& \quad \quad \;\;\; UBD \ge  \sum_{\omega=1}^s obj_{LS_\omega^i} + \sum_{\omega=1}^s  (\pi_\omega^i)^{\textrm{T}} x_0, \quad \forall i \in R^k, \\
& \quad \quad \;\;\; x_{\omega} \in X_{\omega}, {y_{\omega}} \in {Y_{\omega}}. \quad  \\ 
\end{split}
\end{equation}

\begin{equation} \label{eq:LS0}   \tag{LS$_0^k$}
\begin{split} 
& \min_{x_{0}} \; \sum_{\omega=1}^s (\pi_\omega^k)^T x_0 \\ 
& \textrm{s.t.} \quad \; \; \; UBD \ge  \sum_{\omega=1}^s obj_{LS_\omega^i} + \sum_{\omega=1}^s  (\pi_\omega^i)^{\textrm{T}} x_0, \quad \forall i \in R^k, \\ 
& \quad \quad  \quad  \;\;\; x_0 \in X_0.
\end{split}
\end{equation}

\begin{equation}  \label{eq:RPMP} \tag{RPMP$^k$}
\begin{split}
& \min_{\substack{x_0,\theta_1^{[i]},...,\theta_s^{[i]} \\ x_{1},...,x_{s} }} \quad  \sum_{\omega=1}^sc_{\omega}^Tx_{\omega} \\ 
& \textrm{s.t.}  \quad \quad  \; x_{0}  = H_{\omega}x_{\omega}, \quad \forall \omega \in \{1,...,s\}, \\   
& \quad \quad  \quad \; \;\; A_{\omega}x_{\omega} + B_{\omega}\sum_{i \in I^k}\theta_\omega^{[i]}y_{\omega}^{[i]} \le 0, \quad \forall \omega \in \{1,...,s\}, \\   
& \quad \quad \quad \;\;\; \sum_{i \in I^k}\theta_\omega^{[i]} = 1, \quad  \theta_\omega^{[i]} 
\ge 0 , \quad \forall i \in I^k, \quad \forall \omega \in \{1,...,s\},\\
& \quad \quad \quad \; \; \; UBD \ge  \sum_{\omega=1}^s obj_{LS_\omega^i} + \sum_{\omega=1}^s  (\pi_\omega^i)^{\textrm{T}} x_0, \quad \forall i \in R^k, \\
& \quad \quad \quad \;\;\; x_0 \in X_0, \quad x_{\omega} \in X_{\omega}, \quad  \forall \omega \in \{1,...,s\}, 
\end{split}
\end{equation}
Note that the index set $I^k$ includes indices for all constant points $y^{[i]}_\omega$ in Problem  \eqref{eq:RPMP}, and the constant points $y^{[i]}_\omega$ come from all previously solved PP, FP, LS and JRMP.

\subsection{The basic joint decomposition algorithm}

Table \ref{tab:JD1} shows the basic JD algorithm. As described in Section 3.1, a JD iteration always include a LD iteration and sometimes a GBD iteration as well. Whether the GBD iteration is performed at JD iteration $k$ depends on whether LD iteration $k$ improves over LD iteration $k-1$ substantially, i.e., whether $obj_{{LS}^{k}} \geq obj_{{LS}^{k-1}} + \epsilon$. This strategy implies the following result.

\begin{table} [tbp] 
	\caption{The basic joint decomposition algorithm}
	\label{tab:JD1}              
	\begin{tabular} {c}
		\hline
		\\
		\begin{minipage} {\textwidth}
			\centering	     	 		 
			\begin{algorithmic}
				
				\STATE \underline{\textbf{Initialization}} 
				\vspace{0.1cm}
				\STATE
				\begin{enumerate}
					\setlength{\itemsep}{0pt} 
					
					\item[(I.a)] Select $x_0^1, y_1^{[1]}, \cdots, y_s^{[1]}$ that are feasible for Problem \eqref{eq:P}. 
					
					\item[(I.b)] Give termination tolerance $\epsilon > 0$. Let index sets $T^{1}= S^{1}=R^1=\emptyset$, $I^1=\{1\}$, iteration counter $k=1$, $i=1$, $l=1$, bounds $UBD=+\infty$, $LBD=-\infty$.   
				\end{enumerate}
				
				\vspace{0.1cm}
				\STATE \underline{\textbf{LD Iteration}}
				\STATE
				\begin{enumerate}
					\setlength{\itemsep}{0pt}
					\item[(1.a)] Solve Problem \eqref{eq:PP}. If Problem \eqref{eq:PP} is infeasible, solve Problem \eqref{eq:FPw}. Let the solution obtained be $(x_{\omega}^k,y_{\omega}^k)$, and update $i=i+1$, $I^k$=$I^k \cup \{i \}$, $(y^{[i]}_1, \cdots, y^{[i]}_s)=(y^k_1, \cdots, y^k_s)$. 
					
					\item[(1.b)] Solve Problem (BPP$^k_{\omega}$) by fixing $(x_0,y_1,...,y_s)=(x_0^k,y_{1}^k,...,y_{s}^k)$. If (BPP$^k_{\omega}$) is feasible for all $\omega$, generate Benders optimality cuts with the obtained dual solution $\mu_\omega^k$ and $\lambda_\omega^k$, and update  $T^{k+1}=T^k \cup \{k\}$. If $\sum_{\omega=1}^{s} obj_{PP^k_{\omega}}<UBD$, update $UBD=\sum_{\omega=1}^{s} obj_{PP^k_{\omega}}$, and incumbent solution $(x_0^*,x_{1}^*,\cdots, x_{s}^*, y_1^*, \cdots, y_s^*)=(x_0^k,x_{1}^k, \cdots,x_{s}^k, y_1^k, \cdots, y_{s}^k)$. If Problem (BPP$^k_{\omega}$) is infeasible for at least one $\omega$, solve Problem (BFP$^k_{\omega}$). Generate Benders feasibility cuts with the obtained dual solution $\mu_\omega^k$ and $\lambda_\omega^k$, and update  $S^{k+1}=S^k \cup \{k\}$.
					
					\item[(1.c)] Solve Problem \eqref{eq:RPMP}. Let $x_0^k$, $\{\theta_\omega^{[i,k]}\}_{i \in I^k, \omega \in \{1,...,s\}}$ be the optimal solution obtained, and $\pi_1^k,...,\pi_s^k$ be Lagrange multipliers for the NACs. 
					
					\item[(1.d)] Solve Problems \eqref{eq:LSw} and \eqref{eq:LS0}, and let the obtained solution be $(x_{\omega}^k$, $y_{\omega}^k)$, $x_0^k$. If $obj_{LS^k}=\sum_{\omega=1}^{s}obj_{{LS1}_\omega^k}+obj_{{LS_0}^k}>LBD$, update $LBD=obj_{LS^k}$. Generate a Lagrangian cut and update $R^{k+1}=R^{k} \cup \{k\}$. Update $i=i+1$, $I^{k+1} =I^k \cup \{i\} $, $(y^{[i]}_1, \cdots, y^{[i]}_s)=(y^k_1, \cdots, y^k_s)$.
					
					\item[(1.e)] If $UBD \le LBD+\epsilon$, terminate and return the incumbent solution as an $\epsilon$-optimal solution. If $obj_{{LS}^{k}} \geq obj_{{LS}^{k-1}} + \epsilon$, $k=k+1$, go to step (1.a); otherwise $k=k+1$ and go to step (2.a);

				\end{enumerate}

				\STATE \underline{\textbf{GBD Iteration}} 
				\STATE
				\begin{enumerate}
					\setlength{\itemsep}{0pt} 
					
					\item[(2.a)] Solve Problem \eqref{eq:JRMP}, and let the obtained solution be $(x_0^{(l)},y_{1}^{(l)},...,y_{s}^{(1)})$. Update $i=i+1$, $I^{k+1} =I^k \cup \{i\}$, $(y^{[i]}_1, \cdots, y^{[i]}_s)=(y^{(l)}_1, \cdots, y^{(l)}_s)$. If $obj_{RMP^{(l)}}>LBD$, update $LBD=obj_{JRMP^{(l)}}$. 
					
					\item[(2.b)] Solve Problem \eqref{eq:BPP} by fixing $(x_0,y_1,\cdots,y_s)=(x_0^{(l)},y_{1}^{(l)},\cdots,y_{s}^{(l)})$. If \eqref{eq:BPP} is feasible for all $\omega$, generate Benders optimality cuts with the dual solution $\mu_\omega^k$ and $\lambda_\omega^k$, and update $T^{(l+1)}=T^{(l)} \cup \{l\}$. If $\sum_{\omega=1}^{s} obj_{BPP^{(l)}_{\omega}}<UBD$, update $UBD=obj_{BPP^{(l)}}$ and the incumbent solution $(x_0^*,x_{1}^*,\cdots, x_s^*, y_{1}^*, \cdots, y_s^*)=(x_0^{(l)},x_{1}^{(l)},\cdots, x_s^{(l)}, y_{1}^{(l)}), \cdots, y_{s}^{(l)})$.
					If Problem \eqref{eq:BPP} is infeasible for at least one $\omega$, solve Problem \eqref{eq:BFP}. Generate Benders feasibility cuts with the obtained dual solution $\mu_\omega^l$ and $\lambda_\omega^l$, and update $S^{(l+1)}=S^{(l)} \cup \{l\}$.
					
					\item[(2.c)]  If $UBD \le LBD+\epsilon$, terminate and return the incumbent solution as an $\epsilon$-optimal solution; otherwise $l=l+1$, go to step (1.a).

				\end{enumerate}

				
			\end{algorithmic}
			
		\end{minipage} 
		\\
		\\
		\hline
	\end{tabular}
	
\end{table}

\begin{proposition} \label{prop:LDfinite}
The JD algorithm shown in Table \ref{tab:JD1} cannot perform an infinite number of LD iterations between two GBD iterations.   
\end{proposition}

\begin{proof}
The initial point $(x_0^1, y_1^{[1]}, \cdots, y_s^{[1]})$ that are feasible for Problem \eqref{eq:P} can lead to a finite upper bound $UBD$. According to Assumption \ref{ass:boundedness}, all Lagrangian subproblems are bounded, so between two GBD iterations, the first LD iteration leads to a finite $obj_{LS}$, and the subsequent LD iterations increase $obj_{LS}$ by at least $\epsilon>0$ (because otherwise a GBD iteration has to be performed). Therefore, in a finite number LD iterations either $obj_{LS}$ exceeds $UBD-\epsilon$ and the algorithm terminates with an $\epsilon$-optimal solution, or a GBD iteration is performed. This completes the proof. 
\end{proof}

\begin{remark} \label{rmk:initialfeasibility} 
If an initial feasible point for Problem \eqref{eq:P} is not known, the initial feasibility problem \eqref{eq:IFP} can be solved to get a feasible point for \eqref{eq:P} or verify that Problem \eqref{eq:P} is infeasible (when the optimal objective value of Problem \eqref{eq:IFP} is positive). Note that it is easy to find a feasible point of Problem \eqref{eq:IFP}.  
\end{remark}

In the JD algorithm, we use $k$ to index both a JD iteration and a LD iteration, as every JD iteration includes one LD iteration. We use $l$ (together with '()') to index a GBD iteration, and usually $l < k$ because not every JD iteration includes one GBD iteration. We use $i$ (together with '[]') to index the columns generated for constructing Problem \eqref{eq:RPMP}. Next, we establish the finite convergence property of the JD algorithm.

\begin{proposition} \label{prop:BDfinite}
	If set $X_{\omega}$ is polyhedral $\forall \omega \in \{1, \cdots, s\}$, the JD algorithm shown in Table \ref{tab:JD1} cannot perform an infinite number of GBD iterations.
\end{proposition}
\begin{proof}
In this case, the GBD part of the algorithm reduces to BD, and BD is known to have finite termination property \cite{benders1962pps} \cite{lasdon1970spp}. The finite termination property results from: \\
 (a) The Benders master problem \eqref{eq:BRMP} (and therefore \ref{eq:JRMP} as well) requires only a finite number of Benders cuts to equal Problem \eqref{eq:P}, due to linear duality theory; \\
 (b) A same Benders cut cannot be generated twice before the optimality gap is closed.  
\end{proof}

\begin{proposition} \label{prop:x0finite}
	If $X_0 \times Y_1 \times \cdots \times Y_s$ is a finite discrete set, the JD algorithm shown in Table \ref{tab:JD1} cannot perform an infinite number of GBD iterations.
\end{proposition}
\begin{proof}
This result comes from the fact that a point in $X_0 \times Y_1 \times \cdots \times Y_s$ cannot be generated twice before the optimality gap is closed. For more details readers can see Theorem 2.4 of \cite{geoffrion1972gbd}. 
\end{proof}

\begin{proposition} \label{prop:GBDfinite}
The JD algorithm shown in Table \ref{tab:JD1} cannot include an infinite number of GBD iterations at which the Benders primal problem BPP is feasible.  
\end{proposition}
\begin{proof}
A similar proposition has been proved in the context of GBD in \cite{geoffrion1972gbd} (as Theorem 2.5). The central idea of the proof can be used here for JD. 

Suppose the JD algorithm includes an infinite number of GBD iterations at which the Benders primal problem BPP is feasible. Let superscript $(n)$ index these GBD iterations, $\{(\eta_0^{(n)},x_0^{(n)},y_{1}^{(n)},...,y_{s}^{(n)})\}$ be the sequence of optimal solutions of JRMP and $\{(\mu_\omega^{(n)},\lambda_\omega^{(n)})\}$ be the sequence of dual solutions of BPP. Since $\{\eta_0^{(n)}\}$ is nondecreasing and is bounded from above, so a subsequence of it converges to a finite value, say $\eta_0^*$. Due to the compactness of $X_0$, $Y_1, \cdots, Y_s$, a subsequence of $\{(x_0^{(n)},y_{1}^{(n)},...,y_{s}^{(n)})\}$, say, $\{(x_0^{(n_i)},y_{1}^{(n_i)},...,y_{s}^{(n_i)})\}$, converges to $(x_0^*,y_{1}^*,...,y_{s}^*) \in X_0 \times Y_1 \times \cdots \times Y_s$. Solving BPP in this subsequence of GBD iterations can be viewed as point-to-set mappings from points in $X_0 \times Y_1 \times \cdots \times Y_s$ to the relevant Lagrange multiplier sets. From Lemma 2.1 of \cite{geoffrion1972gbd} and Assumption \ref{ass:slater}, such a mapping is uniformly bounded in some open neighborhood of the point it maps from. Let such open neighborhood of $(x_0^*,y_{1}^*,...,y_{s}^*)$ be $N(x_0^*,y_{1}^*,...,y_{s}^*)$, then $\exists t$ such that $\forall n_i > t$, $(x_0^{(n_i)},y_{1}^{(n_i)},...,y_{s}^{(n_i)}) \in N(x_0^*,y_{1}^*,...,y_{s}^*)$, and then the relevant subsequence of Lagrange multipliers is bounded, which must contain a subsequence converging to $\{\mu_\omega^\star,\lambda_\omega^\star\}$. Therefore, there exists a subsequence of $\{(\eta_0^{(n)},x_0^{(n)},y_{1}^{(n)},...,y_{s}^{(n)},\mu_\omega^{(n)},\lambda_\omega^{(n)})\}$, say, $\{(\eta_0^{(m)},x_0^{(m)},y_{1}^{(m)},...,y_{s}^{(m)},\mu_\omega^{(m)},\lambda_\omega^{(m)})\}$, which converges to $\{(\eta_0^*,x_0^*,y_{1}^*,...,y_{s}^*,\mu_\omega^*,\lambda_\omega^*)\}$. 

Consider any GBD iteration $m>1$ in this convergent subsequence. Let $UBD$ and $LBD$ be the upper and lower bounds after this GBD iteration, then
\begin{equation*}
obj_{BPP^{(m-1)}} \geq UBD,
\end{equation*}
\begin{equation*}
LBD \geq \eta^{(m)},
\end{equation*}
and that the JD algorithm does not terminate after GBD iteration $m$ implies
\begin{equation*}
UBD > LBD + \epsilon,
\end{equation*}
therefore 
\begin{equation} \label{eq:finite1}
obj_{BPP^{(m-1)}} > \eta^{(m)} + \epsilon. 
\end{equation}
According to how JRMP is constructed, 
\begin{equation} \label{eq:finite2}
\begin{split}
\eta^{(m)} \geq &  obj_{BPP^{(m-1)}} +  \\
&  \sum_{\omega=1}^{s} \left[(\lambda_\omega^{(m-1)})^{\textrm{T}} B_{\omega}(y_{\omega}^{(m)}-y_{\omega}^{(m-1)})+  ({ \mu_{\omega}^{(m-1)}})^{\textrm{T}}\left( x_0^{(m)} - x_0^{(m-1)} \right) \right].
\end{split}
\end{equation}
Equations \eqref{eq:finite1} and \eqref{eq:finite2} imply that
\begin{equation} \label{eq:finite3}
0 > \sum_{\omega=1}^{s} \left[(\lambda_\omega^{(m-1)})^{\textrm{T}} B_{\omega}(y_{\omega}^{(m)}-y_{\omega}^{(m-1)})+ (\mu_{\omega}^{(m-1)})^{\textrm{T}}\left( x_0^{(m)} - x_0^{(m-1)} \right) \right] + \epsilon.
\end{equation}
However, when $m$ is sufficiently large, $y_{\omega}^{(m)}-y_{\omega}^{(m-1)}$ and $x_0^{(m)} - x_0^{(m-1)}$ are sufficiently close to 0 while $\mu_{\omega}^{(m-1)}$ and $\lambda_\omega^{(m-1)}$ are sufficiently close to limit points $\mu_\omega^*$ and $ \lambda_\omega^*$, so the right-hand-side of Equation \eqref{eq:finite3} is a positive value (as $\epsilon>0$). This contradiction implies that the JD algorithm cannot include an infinite number of GBD iterations at which BPP is feasible.

\end{proof}

\begin{theorem} \label{thm:JDfinite}
With an initial feasible point, the JD algorithm shown in Table \ref{tab:JD1} terminates in a finite number of iterations with an $\epsilon$-optimal solution, if one the following three conditions is satisfied: \\
(a) Set $X_{\omega}$ is polyhedral $\forall \omega \in \{1, \cdots, s\}$. \\
(b) Set $X_0 \times Y_1 \times \cdots \times Y_s$ is finite discrete. \\
(c) There are only a finite number of GBD iterations at which the Benders primal problem BPP is infeasible. 
\end{theorem}

\begin{proof}
From Proposition \ref{prop:LDfinite}, the JD algorithm can only include a finite number of LD iterations. From Propositions \ref{prop:BDfinite} and \ref{prop:x0finite}, when condition (a) or (b) is satisfied, the JD algorithm can only include a finite number of BD iterations. From Proposition \ref{prop:GBDfinite}, the JD algorithm can only have a finite number of GBD iterations at which the Benders primal problem BPP is feasible, and together with condition (c), it implies that the JD algorithm can only include a finite number of BD iterations. Therefore, if one of the three conditions is satisfied, the JD algorithm can only include a finite number LD and BD iterations before termination. 

On the other hand, according to Proposition \ref{prop:JRMP}, the JD algorithm never excludes an optimal solution. This together with the termination criterion ensures that the solution returned is $\epsilon$-optimal.
\end{proof}

\begin{remark}
Condition (c) in Theorem \ref{thm:JDfinite} is actually not a very restrictive condition, because we can always "soften" the complicating constraints in Problem \eqref{eq:P} (i.e., penalize the violation of these constraints in the objective function) so that Problem \eqref{eq:BPP} is always feasible. 
\end{remark}

\section{Enhancements to joint decomposition} \label{sec:sec4}

The solution of Problem \eqref{eq:JRMP} is the bottleneck of the JD algorithm, even considering that the problem is solved only when necessary. Problem \eqref{eq:JRMP} is challenging due to two major reasons. One is that the number of complicating variables in Problem \eqref{eq:JRMP} is dependent on the number of scenarios, so the size of Problem \eqref{eq:JRMP} is large (although smaller than the original problem). The other is that the number of constraints in the problem grows with the JD iteration; in other words, Problem \eqref{eq:JRMP} becomes more and more challenging as JD progresses. In this section, we introduce two ways to mitigate the difficulty in solving Problem \eqref{eq:JRMP}: 
\begin{enumerate}
\item To solve a convex relaxation of Problem \eqref{eq:JRMP} before solving Problem \eqref{eq:JRMP}. If the solution of the convex relaxation can improve the lower bound, then skip solving Problem \eqref{eq:JRMP}.

\item To perform domain reduction iteratively in JD in order to keep reducing the ranges of the complicating variables. This way, the convex relaxation of Problem \eqref{eq:JRMP} is progressively tightened and Problem \eqref{eq:JRMP} itself does not become much harder as the algorithm progresses. 
\end{enumerate}

In addition, domain reduction for the complicating variables can make other nonconvex JD subproblems easier, including Problems \eqref{eq:LSw} and \eqref{eq:PP}. Domain reduction for the linking variables can also tighten the Lagrangian relaxation gap \cite{Caroe199937}; in extreme cases, the Lagrangian relaxation gap can diminish and there is no need to solve Problem \eqref{eq:JRMP} in JD to close the optimality gap. Note that we do not perform domain reduction for non-complicating variables, because normally reducing ranges on these variables do not help much to tighten convex relaxations and ease the solution of nonconvex subproblems.

\subsection{Convex relaxation and domain reduction}

The convex relaxation of Problem \eqref{eq:JRMP} is a valid lower bounding problem for Problem \eqref{eq:JRMP} and consequently for Problem \eqref{eq:P} as well. It can be written as: 
\begin{equation} \label{eq:JRMPR}  \tag{JRMPR$^{(l)}$}
\begin{split}
&\min_{\substack{x_0,\eta_0,\eta_1,...,\eta_s \\ y_{1},...,y_{s}}} \quad \eta_0 \\ 
& \textrm{s.t.} \;\;  \eta_0 \ge \sum_{\omega=1}^{s}\eta_\omega, \\ 
& \quad \quad  \eta_\omega  \ge  obj_{\textrm{BPP}_\omega^{(j)}} + ( \lambda_\omega^{(j)})^{\textrm{T}} B_{\omega}(y_{\omega}-y_{\omega}^{(j)})+  ( \mu_\omega^{(j)})^{\textrm{T}}\left( x_0 - x_0^{(j)} \right), \\
& \qquad \qquad \qquad \qquad \qquad \qquad \qquad \qquad  \forall \omega \in \{1,...,s\}, \quad \forall  j \in {T^{(l)}}, \\ 
& \quad \quad  0  \ge obj_{\textrm{BFP}_\omega^{(j)}} + (\lambda_\omega^{(j)})^{\textrm{T}} B_{\omega}(y_{\omega}-y_{\omega}^{(j)})+ ( \mu_\omega^{(j)})^{\textrm{T}}\left( x_0 - x_0^{(j)} \right), \\
& \qquad \qquad \qquad \qquad \qquad \qquad \qquad \qquad \forall \omega \in \{1,...,s\}, \quad  
 \forall  j \in {S^{(l)}}, \\ 
& \quad \quad \eta_\omega  \ge  obj_{\textrm{BPP}_\omega^j} + ( \lambda_\omega^j)^{\textrm{T}} B_{\omega}(y_{\omega}-y_{\omega}^j)+  ({ \mu_{\omega}^{j}})^{\textrm{T}}\left( x_0 - x_0^{j} \right), \\
& \qquad \qquad \qquad \qquad \qquad \qquad \qquad \qquad
\forall \omega \in \{1,...,s\}, \quad \forall  j \in {T^k}, \\ 
& \quad \quad  0  \ge obj_{\textrm{BFP}_\omega^j} + (\lambda_\omega^j)^{\textrm{T}} B_{\omega}(y_{\omega}-y_{\omega}^{j})+ ( \mu_\omega^j)^{\textrm{T}}\left( x_0 - x_0^{j} \right), \\
& \qquad \qquad \qquad \qquad \qquad \qquad \qquad \qquad \forall \omega \in \{1,...,s\}, \quad  \forall  j \in S^k, \\ 
& \quad \quad \eta_0 \le UBD, \\ 
& \quad \quad \eta_0 \ge LBD, \\    
& \quad \quad \eta_\omega  \ge  obj_{\textrm{LS}_\omega^i} +  ({ \pi_{\omega}^{i}})^{\textrm{T}}x_0, \quad \forall \omega \in \{1,...,s\}, \quad \forall  i \in {R^k}, \\ 
& \quad \quad x_0\in \hat{X}_0,\quad y_{\omega} \in \hat{Y}_{\omega}, \quad \forall \omega \in \{1,...,s\}.  \\ 
\end{split}
\end{equation}
Here $\hat{X}_0$ and $\hat{Y}_{\omega}$ denote the convex relaxations of ${X}_0$ and ${Y}_{\omega}$. Let $obj_{{JRMPR}^{(l)}}$ be the optimal objective of Problem \eqref{eq:JRMPR}.  

Since Problem \eqref{eq:JRMPR} is also a valid convex relaxation of Problem \eqref{eq:P}, the solution of Problem \eqref{eq:JRMPR} can be exploited to eliminate the parts of variable ranges that cannot include an optimal solution of Problem \eqref{eq:P}, using marginal based domain reduction method. This method was first proposed in \cite{ryoo1996bat} (and it was called range reduction therein). The following proposition lays the foundation of marginal based domain reduction for complicating variables $y_\omega$ in JD, which results directly from Theorem 2 in \cite{ryoo1996bat}.

\begin{proposition} \label{prop:rangered1}
	Consider the following bounds on $y_{\omega,j}$ ($\forall \omega \in \{1, \cdots, s\}, \;\; \forall j \in \{1, \cdots, n_y\}$):
	\begin{align*}
	y_{\omega,j} - y_{\omega,j}^{up} \le 0 ,  \\
	y_{\omega,j}^{lo} - y_{\omega,j} \le 0 , 	
	\end{align*}
	whose Lagrange multipliers obtained at the solution of Problem \eqref{eq:JRMPR} are $u_{\omega,j}$, $v_{\omega,j}$. Let $\mathbb{J}^{(l)}_{1,\omega}$ include indices of upper bounds whose $u_{\omega,j}$ are nonzero, and  $\mathbb{J}^{(2)}_{1,\omega}$ include indices of lower bounds whose $v_{\omega,j}$ are nonzero, then the following constraints do not exclude an optimal solution of \eqref{eq:P}:
	\begin{equation*} 
	\begin{split}
	y_{\omega,j} & \ge  y_{\omega,j}^{up}- \frac{(UBD-obj_{{JRMPR}^{(l)}})}{u_{\omega,j}}, \quad \forall j \in \mathbb{J}^{(l)}_{1,\omega} , \;\; \forall \omega \in \{1,...,s\}, \\
	 y_{\omega,j} & \le  y_{\omega,j}^{lo}+ \frac{(UBD-obj_{{JRMPR}^{(l)}})}{v_{\omega,j}}, \quad \forall j \in \mathbb{J}^{(l)}_{2,\omega} , \; \forall \omega \in \{1,...,s\}. 		
	\end{split}
	\end{equation*}
 
	\end{proposition}

	The following proposition states a similar result for the linking variables $x_0$:
	
\begin{proposition} \label{prop:rangered2}
Consider the following bounds on $x_{0,j}$ ($\forall j \in \{1, \cdots, n_0\}$):
	\begin{align*}
	x_{0,j} - x_{0,j}^{up} \le 0 ,  \\
	x_{0,j}^{lo} - x_{0,j} \le 0 , 	
	\end{align*}
	whose Lagrange multipliers obtained at the solution of Problem \eqref{eq:JRMPR} are $u_{0,j}$, $v_{0,j}$. Let $\mathbb{J}^{(l)}_{1,0}$ include indices of upper bounds whose $u_{0,i}$ are nonzero, and $\mathbb{J}^{(l)}_{2,0}$ include indices of lower bounds whose $v_{0,i}$ are nonzero, then the following constraints do not exclude an optimal solution of \eqref{eq:P}:
\begin{equation*}
\begin{split}
x_{0,j} &  \ge  x_{0,j}^{up}- \frac{(UBD-obj_{JRMPR})}{u_{0,j}}, \quad \forall j \in \mathbb{J}^{(l)}_{1,0} \\
x_{0,j} & \le  x_{0,j}^{lo}+\frac{(UBD-obj_{{JRMPR}^{(l)}})}{v_{0,j}}, \quad \forall j \in \mathbb{J}^{(l)}_{2,0}
\end{split}
\end{equation*}	
\end{proposition}

According to Propositions \ref{prop:rangered1} and \ref{prop:rangered2}, the bounds of nonconvex and linking variables can be updated via the following range reduction calculation: 
\begin{equation} \label{eq:MDR} \tag{MDR$^{(l)}$}
\begin{split}
y_{\omega,j}^{up} & =\min \left\{y_{\omega,j}^{up}, \;\; y_{\omega,j}^{lo} + \frac{G^{(l)}}{u_{\omega,j}} \right\}, \quad \forall j \in \mathbb{J}^{(l)}_{1,\omega}, \;\; \forall \omega \in \{1,...,s\},  \\
y_{\omega,j}^{lo} & =\max \left\{y_{\omega,j}^{lo}, \;\; y_{\omega,j}^{up} - \frac{G^{(l)}}{v_{\omega,j}} \right\}, \quad \forall j \in \mathbb{J}^{(l)}_{2,\omega}, \;\; \forall \omega \in \{1,...,s\},   \\
x_{0,j}^{up} & =\min \left\{x_{0,j}^{up}, \;\; x_{0,j}^{lo}+ \frac{G^{(l)}}{u_{0,j}} \right\}, \quad \forall j \in \mathbb{J}^{(l)}_{1,0}, \\
x_{0,j}^{lo} & =\max \left\{ x_{0,j}^{lo}, \;\; x_{0,j}^{up}- \frac{G^{(l)}}{v_{0,j}} \right\}, \quad \forall j \in \mathbb{J}^{(l)}_{2,0}, \\
\end{split}
\end{equation}
where $G^{(l)}=UBD-obj_{{RMPCR}^{(l)}}$.

The effectiveness of marginal based domain reduction relies on how many bounds are active, the magnitude of Lagrange multipliers of active bounds at the solution of \ref{eq:JRMPR}, and how often \ref{eq:JRMPR} is solved. In order to achieve effective domain reduction more consistently, we also introduce optimization based domain reduction in JD. Optimization based domain reduction, or called bound contraction or bound tighening \cite{Zamora1999} \cite{Maranas1996}, is to maximize or minimize a single variable over a convex relaxation of the feasible set of the original problem. For example, if we are to estimate the upper bound of a linking variable $x_{0,j}$ at JD iteration $k$, we can solve the following optimization problem:
\begin{equation} \label{eq:ODR0} \tag{ODRStd$_i^k$}
\begin{split}
& \max_{\substack{ x_0,x_{1},...,x_{s} \\ y_{1},,...,y_{s}}} \;  x_{0,i}\\ 
& \textrm{s.t.}\;\; {x_{0}}  = H_{\omega}x_{\omega}, \quad \forall {\omega} \in \{1,...,s\}, \\  
& \quad \;\;\; A_{\omega}{x_{\omega}} + B_{\omega}y_{\omega} \le 0, \quad \forall {\omega} \in \{1,...,s\},\\ 
& \quad \;\;\;  \sum_{\omega=1}^s  {c_{\omega}^{\textrm{T}}}{x_{\omega}}\le UBD, \\
& \quad \;\;\; {x_{0}} \in {X_{0}^k}, \\
& \quad \;\;\; {x_{\omega}} \in {X_{\omega}}, \; {y_{\omega}} \in {\hat{Y}_{\omega}^k}, \quad \forall {\omega} \in \{1,...,s\}. \\ 
\end{split}
\end{equation}
The third group of constraints in Problem \eqref{eq:ODR0} utilizes the known upper bound of \eqref{eq:P} to tighten the convex relaxation, but it cannot be included in Problem \eqref{eq:ODR0} when $UBD$ is not available (e.g., before a feasible solution of \eqref{eq:P} is known). We now index sets $X_{0}$, $\hat{Y}_{\omega}$ with the JD iteration number $k$, as these sets may change after the domain reduction calculations. 

Problem \eqref{eq:ODR0} represents the standard optimization based domain reduction formulation, but it can be further enhanced in the JD algorithm, via the incorporation of valid cuts derived from other JD subproblems. First, we can add the following constraint:
\[ \sum_{\omega=1}^s  {c_{\omega}^{\textrm{T}}}{x_{\omega}}\ge LBD. \]
This constraint is redundant in the classical branch-and-bound based global optimization, as $LBD$ is obtained via convex relaxation as well. In JD, $LBD$ is obtained via Lagrangian subproblems and JD relaxed master problems, which may be tigher than convex relaxations of the original problem, so this constraint may enhance Problem \eqref{eq:ODR0}. Second, we can include constraints \eqref{eq:star} (that are drived from Problem \eqref{eq:JRMP}). Therefore, we can write the enhanced optimization based domain reduction formulation as:
\begin{equation} \label{eq:ODR} \tag{ODR$_i^k$}
\begin{split}
& \min_{\substack{ x_0,x_{1},...,x_{s} \\ y_{1},,...,y_{s}}} / \max_{\substack{ x_0,x_{1},...,x_{s} \\ y_{1},,...,y_{s}}}\;  x_{0,i}\\ 
& \textrm{s.t.}\;\; {x_{0}}  = H_{\omega}x_{\omega}, \quad \forall {\omega} \in \{1,...,s\}, \\  
& \quad \;\;\; A_{\omega}{x_{\omega}} + B_{\omega}y_{\omega} \le 0, \quad \forall {\omega} \in \{1,...,s\},\\  
& \quad \;\;\; \sum_{\omega=1}^s  {c_{\omega}^{\textrm{T}}}{x_{\omega}}\le UBD, \\ 
& \quad \;\;\; \sum_{\omega=1}^s {c_{\omega}^{\textrm{T}}}{x_{\omega}}\ge LBD, \\ 
& \quad \;\;\;  UBD \ge  \sum_{\omega=1}^s obj_{LS_\omega^i} + \sum_{\omega=1}^s  (\pi_\omega^i)^{\textrm{T}} x_0, \quad \forall i \in R^k, \\
& \quad \;\;\; {x_{0}} \in {X_{0}^k}, \\
& \quad \;\;\;  {x_{\omega}} \in {X_{\omega}}, \; {y_{\omega}} \in {\hat{Y}_{\omega}^k}, \quad \forall {\omega} \in \{1,...,s\}. \\ 
\end{split}
\end{equation}
If we are to estimate an upper bound, then Problem \eqref{eq:ODR} is a maximization problem; otherwise, Problem \eqref{eq:ODR} is a minimization problem. 

Although Problem \eqref{eq:ODR} is convex, it can have a very large size because its size grows with the number of scenarios. Therefore, we proposed to solve Problem \eqref{eq:ODR} for $x_0$ but not for $y_\omega$. Actually, we can see in the case study section that optimization based domain reduction is time consuming even when we only solve Problem \eqref{eq:ODR} for $x_0$.

\subsection{The enhanced joint decomposition method}

Figure \ref{fig:figure2} shows the framework of the JD method that includes solving convex relaxation, Problem \eqref{eq:JRMPR}, bound tightening for $x_0$ and the domain reduction calculations. In this framework, optimization based domain reduction is performed at the beginning of the algorithm and in every LD iteration (right before the solution of nonconvex Lagrangian subproblems). Convex relaxation, Problem \eqref{eq:JRMPR} is solved before solving Problem \eqref{eq:JRMP}, and after solving Problem \eqref{eq:JRMPR}, marginal based domain reduction is performed. Problem \eqref{eq:JRMP} is not solved if Problem \eqref{eq:JRMPR} can improve the lower bound significantly; this strategy can postpone solving Problem \eqref{eq:JRMP} to a later time, so that the ranges of $x_0$ can be reduced as much as possible when a Problem \eqref{eq:JRMP} has to be solved. The detailed algorithm for the enhanced JD is shown in Table \ref{tab:JD2}.

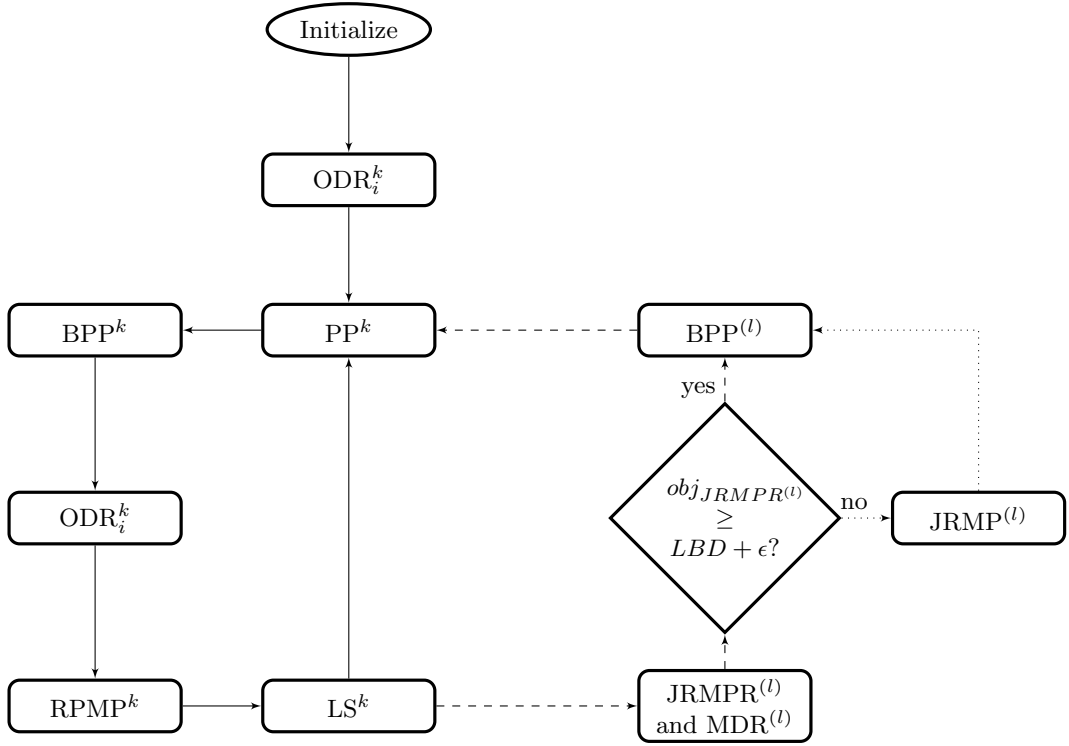
\begin{figure}[thb!]
	\begin{center}
		\centering	
		\begin{tikzpicture}[node distance = 2cm, auto]
		\node [cloud] (init0) {Initialize};
		\node [block,below of=init0] (init1) {ODR$_i^k$};
		\node [block,below of=init1] (init) {PP$^k$};
		\node [block, left of=init,xshift=-4em,yshift=0em] (identify1) {BPP$^k$};
		\node [block, below of=identify1,node distance=2.5cm] (Contract) {ODR$_i^k$};
		\node [block, below of=init,node distance=5cm] (identify2) {LS$^k$};
		\node [block, left of=identify2, xshift=-4em,yshift=0em] (item) {RPMP$^k$};
		\node [block, right of=init, node distance=5cm] (upper) {BPP$^{(l)}$};
		\node [decision, below of=upper] (decide) {$obj_{JRMPR^{(l)}}$ $\ge$ $LBD+\epsilon$?};
		\node [block, below of=decide, node distance=2.5cm] (real) {JRMPR$^{(l)}$ and MDR$^{(l)}$};
		\node [block, right of=decide, xshift=4em] (real1) {JRMP$^{(l)}$};
		\path [line] (init0) -- (init1);
		\path [line] (init1) -- (init);
		\path [line] (init) -- (identify1);
		\path [line] (identify2) -- (init);
		\path [line] (identify1) --  (Contract);
		\path [line] (Contract) --  (item);
		\path [line] (item) --  (identify2);
		
		\path [line,dashed] (upper) --  (init);
		\path [line,dashed] (identify2) -- (real);
		
		\path [line,dashed] (real) -- (decide);
		\path [line,dashed] (decide) -- node [near start] {yes} (upper);
		\path [line,dotted] (decide) -- node [near start] {no} (real1);
		\path [line,dotted] (real1) |- (upper);	
		\end{tikzpicture}
		\vspace{0.2cm}	
		\caption{The enhanced joint decomposition framework}      \label{fig:figure2}	
	\end{center}	
\end{figure}

\begin{table} [tbp] 
	\caption{Enhanced joint decomposition method - Enhancement is in bold font}
	\label{tab:JD2}              
	\begin{tabular} {c}
		\hline
		\\
		\begin{minipage} {\textwidth}
			\centering	     	 		 
			\begin{algorithmic}
				
				\STATE \underline{\textbf{Initialization}} 
				\vspace{0.1cm}
				\STATE
				\begin{enumerate}
					\setlength{\itemsep}{0pt} 
					
					\item[(I.a)] Select $x_0^1, y_1^{[1]}, \cdots, y_s^{[1]}$ that are feasible for Problem \eqref{eq:P}. 
					
					\item[(I.b)] Give termination tolerance $\epsilon > 0$. Let index sets $T^{1}= S^{1}=R^1=\emptyset$, $I^1=\{1\}$, iteration counter $k=1$, $i=1$, $l=1$, bounds $UBD=+\infty$, $LBD=-\infty$.   
					
					\item[(I.c)] \textbf{Solve Problem \eqref{eq:ODR} to update bounds of all $x_{0,i}$. } 
					
				\end{enumerate}
				
				\vspace{0.1cm}
				\STATE \underline{\textbf{LD Iteration}}
				\STATE
				\begin{enumerate}
					\setlength{\itemsep}{0pt}
					\item[(1.a)] Solve Problem \eqref{eq:PP}. If Problem \eqref{eq:PP} is infeasible, solve Problem \eqref{eq:FPw}. Let the solution obtained be $(x_{\omega}^k,y_{\omega}^k)$, and update $i=i+1$, $I^k$=$I^k \cup \{i \}$, $(y^{[i]}_1, \cdots, y^{[i]}_s)=(y^k_1, \cdots, y^k_s)$. 
					
					\item[(1.b)] Solve Problem (BPP$^k_{\omega}$) by fixing $(x_0,y_1,...,y_s)=(x_0^k,y_{1}^k,...,y_{s}^k)$. If (BPP$^k_{\omega}$) is feasible for all $\omega$, generate Benders optimality cuts with the obtained dual solution $\mu_\omega^k$ and $\lambda_\omega^k$, and update  $T^{k+1}=T^k \cup \{k\}$. If $\sum_{\omega=1}^{s} obj_{PP^k_{\omega}}<UBD$, update $UBD=\sum_{\omega=1}^{s} obj_{PP^k_{\omega}}$, and incumbent solution $(x_0^*,x_{1}^*,\cdots, x_{s}^*, y_1^*, \cdots, y_s^*)=(x_0^k,x_{1}^k, \cdots,x_{s}^k, y_1^k, \cdots, y_{s}^k)$. If Problem (BPP$^k_{\omega}$) is infeasible for at least one $\omega$, solve Problem (BFP$^k_{\omega}$). Generate Benders feasibility cuts with the obtained dual solution $\mu_\omega^k$ and $\lambda_\omega^k$, and update  $S^{k+1}=S^k \cup \{k\}$.
					
					\item[(1.c)] \textbf{Solve Problem \eqref{eq:ODR} to update bounds of all $x_{0,i}$. } 
					
					\item[(1.d)] Solve Problem \eqref{eq:RPMP}. Let $x_0^k$, $\{\theta_\omega^{[i,k]}\}_{i \in I^k, \omega \in \{1,...,s\}}$ be the optimal solution obtained, and $\pi_1^k,...,\pi_s^k$ be Lagrange multipliers for the NACs. 
					
					\item[(1.e)] Solve Problems \eqref{eq:LSw} and \eqref{eq:LS0}, and let the obtained solution be $(x_{\omega}^k$, $y_{\omega}^k)$, $x_0^k$. If $obj_{LS^k}=\sum_{\omega=1}^{s}obj_{{LS}_\omega^k}+obj_{{LS_0}^k}>LBD$, update $LBD=obj_{LS^k}$. Generate a Lagrangian cut and update $R^{k+1}=R^{k} \cup \{k\}$. Update $i=i+1$, $I^{k+1} =I^k \cup \{i\} $, $(y^{[i]}_1, \cdots, y^{[i]}_s)=(y^k_1, \cdots, y^k_s)$.
					
					\item[(1.f)] If $UBD \le LBD+\epsilon$, terminate and return the incumbent solution as an $\epsilon$-optimal solution. If $obj_{{LS}^{k}} \geq obj_{{LS}^{k-1}} + \epsilon$, $k=k+1$, go to step (1.a); otherwise $k=k+1$ and go to step (2.a);

				\end{enumerate}

				\STATE \underline{\textbf{GBD Iteration}} 
				\STATE
				\begin{enumerate}
					\setlength{\itemsep}{0pt} 
					
					\item[(2.a)] \textbf{Solve Problem \eqref{eq:JRMPR}, and then perform marginal based domain reduction \eqref{eq:MDR}. If $obj_{JRMPR^{(l)}} \ge LBD + \epsilon$, let the obtained solution be $(x_0^{(l)},y_{1}^{(l)},...,y_{s}^{(1)})$, update $LBD=obj_{JRMPR^{(l)}}$, $i=i+1$, $I^{k+1} =I^k \cup \{i\}$, $(y^{[i]}_1, \cdots, y^{[i]}_s)=(y^{(l)}_1, \cdots, y^{(l)}_s)$, go to step (2.c). Otherwise, go to set (2.b). } 
					
					\item[(2.b)] Solve Problem \eqref{eq:JRMP}, and let the obtained solution be $(x_0^{(l)},y_{1}^{(l)},...,y_{s}^{(1)})$. Update $i=i+1$, $I^{k+1} =I^k \cup \{i\}$, $(y^{[i]}_1, \cdots, y^{[i]}_s)=(y^{(l)}_1, \cdots, y^{(l)}_s)$. If $obj_{RMP^{(l)}}>LBD$, update $LBD=obj_{JRMP^{(l)}}$. 
					
					\item[(2.c)] Solve Problem \eqref{eq:BPP} by fixing $(x_0,y_1,\cdots,y_s)=(x_0^{(l)},y_{1}^{(l)},\cdots,y_{s}^{(l)})$. If \eqref{eq:BPP} is feasible for all $\omega$, generate Benders optimality cuts with the dual solution $\mu_\omega^k$ and $\lambda_\omega^k$, and update $T^{(l+1)}=T^{(l)} \cup \{l\}$. If $\sum_{\omega=1}^{s} obj_{BPP^{(l)}_{\omega}}<UBD$, update $UBD=obj_{BPP^{(l)}}$ and the incumbent solution $(x_0^*,x_{1}^*,\cdots, x_s^*, y_{1}^*, \cdots, y_s^*)=(x_0^{(l)},x_{1}^{(l)},\cdots, x_s^{(l)}, y_{1}^{(l)}), \cdots, y_{s}^{(l)})$.
					If Problem \eqref{eq:BPP} is infeasible for at least one $\omega$, solve Problem \eqref{eq:BFP}. Generate Benders feasibility cuts with the obtained dual solution $\mu_\omega^l$ and $\lambda_\omega^l$, and update $S^{(l+1)}=S^{(l)} \cup \{l\}$.
					
					\item[(2.d)]  If $UBD \le LBD+\epsilon$, terminate and return the incumbent solution as an $\epsilon$-optimal solution; otherwise $l=l+1$, go to step (1.a).

				\end{enumerate}

				
			\end{algorithmic}
			
		\end{minipage} 
		\\
		\\
		\hline
	\end{tabular}
	
\end{table}

\begin{theorem} \label{thm:JD2finite}
	The decomposition algorithm described in Table \ref{tab:JD2} terminates in a finite number of steps with an $\epsilon$-optimal solution of Problem \eqref{eq:P}, if one the following three conditions is satisfied: \\
	(a) Set $X_{\omega}$ is polyhedral $\forall \omega \in \{1, \cdots, s\}$. \\
	(b) Set $X_0 \times Y_1 \times \cdots \times Y_s$ is finite discrete. \\
	(c) There are only a finite number of GBD iterations at which the Benders primal problem BPP is infeasible.  
\end{theorem}

\begin{proof}
This can be proved by showing that, solving Problem \eqref{eq:JRMPR} in every GBD iteration in JD and including domain reduction calculations do not invalidate the finite termination to an $\epsilon$-optimal solution. 

First, we can show that there cannot be an infinite number of GBD iterations at which Problem \eqref{eq:JRMPR} is solved but Problem \eqref{eq:JRMP} is not solved. Consider a GBD iteration at which Problem \eqref{eq:JRMPR} is solved but Problem \eqref{eq:JRMP} is not solved, then Problem \eqref{eq:JRMPR} is not unbounded (because otherwise Problem \eqref{eq:JRMP} needs to be solved) and the lower bound $LBD$ is finite. The upper bound $UBD$ is also finite (because an initial feasible solution exists). Therefore, it is not possible that $LBD$ can be improved by $\epsilon>0$ for an infinite number of GBD iterations, so there cannot be an infinite number of GBD iterations at which Problem \eqref{eq:JRMPR} is solved but Problem \eqref{eq:JRMP} is not solved. According to the proof of Theorem \ref{thm:JDfinite}, JD can only include a finite number of LD iterations, and a finite number of GBD iterations at which Problem \eqref{eq:JRMP} is solved, if one of the three listed conditions are satisfied.  

Second, domain reduction reduces the ranges of $x_{0}$ and $y_{1},...,y_{s}$ but does not exclude any optimal solution from the reduced ranges. So the Lagrangian relaxation problems and JD relaxation master problems are still valid lower bounding problems and they cannot cut off any optimal solution. 
 
\end{proof}

\section{Case Studies} \label{sec:sec5}

The purpose of the case studies is to demonstrate the potential computational advantages of the proposed joint decomposition method for problems exhibiting the decomposable structure of \eqref{eq:P0}, especially when off-the-shelf solvers cannot effectively exploit the problem structure. We consider two case study problems here, which are both scenario-based two-stage stochastic nonconvex MINLPs arising from integrated design and operation under uncertainty. 

\subsection{Case study problems} \label{subsec:5.1}

\textbf{Case Study A} - This problem is a variant of the stochastic Haverly pooling problem \cite{li2010spp}, which was originally developed based on the classical Haverly pooling problem \cite{haverly1978sob} \cite{haverly1979bor}. Figure \ref{fig:SHaverly} shows the superstructure of the pooling system to be developed. The circles denote four sources that supply intermediate gasoline products with different sulfur percentages and costs, the ellipse denotes a blender (or called a pool) at which some intermediate products can be blended, and the rectangles denote product sinks at which the final products are blended. The goal of optimization is to minimize the negative profit of the system by determining: (1) Whether the pool and the two product sinks are to be developed in the system; (2) The capacities of the sources and the pipelines. The stochastic pooling model of the problem can be found in Appendix B. Two uncertain parameters, percentage of sulfur in source 4 and upper limit on the demand at sink 1, were considered. They were assumed to follow independent normal distributions, with means of 2.5 and 180 and standard deviations of 0.08 and 10, respectively. Other parameters used in the problem can be found in \cite{li2010spp}. For this problem, $x_0$ contains 3 binary variables and 13 continuous variables, $x_\omega$ contains $7s$ continuous variables and $y_\omega$ contains $14s$ continuous variables, where $s$ stands for the total number of scenarios. In the case study, each uncertain parameter was sampled for 5, 6, 7, 8, 9 and 10 scenario values, via the sampling rule described in \cite{li2010spp}, and this led to problem instances with 25, 36, 49, 64, 81 and 100 scenarios.

\begin{figure}[tb]
	\begin{center}
		\centering	
		\includegraphics[width=0.65 \textwidth]{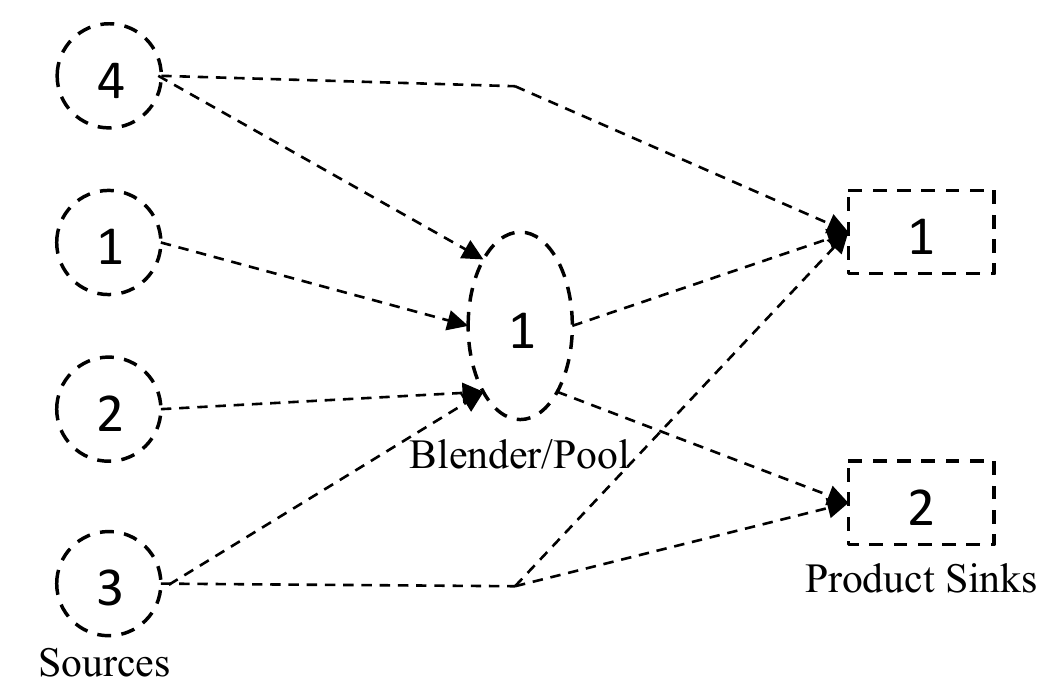}
		\caption{Superstructure of case study A problem}      
		\label{fig:SHaverly}	
	\end{center}	
\end{figure}

\textbf{Case Study B} - This problem is a variant of the Sarawak Gas Production System (SGPS) design problem \cite{selot2008sop}, and the original form of the design problem appeared in \cite{li2010spp}. Figure \ref{fig:SGPS} shows the superstructure of the SGPS system under consideration, where the circles represent gas fields (sources), ellipses represent offshore gas platforms (pools) at which gas flows from different gas fields are mixed and split, rectangles represent onshore liquefied natural gas (LNG) plants (product terminals). Symbols with solid lines represent the part of the system that is already developed, and symbols with dashed lines represent the superstructure of the part of the system that needs to be designed in the problem. The goal of optimization is to maximize expected net present value while satisfying specifications for gas qualities at the LNG plants in the presence of uncertainty. There are two uncertain parameters, i.e., the quality of CO$_2$ at gas field M1 and upper limit on the demand at LNG plant 2. They were assumed to follow independent normal distributions with means of 3.34\% and 2155 Mmol/day and standard deviations of 1\% and 172.5 Mmol/day, respectively. In the case study, each uncertain parameter was sampled for 5, 6, 7, 8, 9 and 10 scenario values, via the same sampling rule described in \cite{li2010spp}, which led to problem instances with 25, 36, 49, 64, 81 and 100 scenarios. The problem was also formulated following the new stochastic pooling model provided in Appendix B. In the resulting formulation, $x_0$ contains 5 binary variables and 29 continuous variables. The 5 binary variables are to determine whether gas fields HL, SE, M3, M1 and JN are to be developed, and the 29 continuous variables are the capacities of other units to be developed. $x_\omega$ contains $8s$ variables and $y_\omega$ contains $85s$ variables, where $s$ stands for the total number of scenarios.

\begin{figure}[tb]
	\begin{center}
		\centering	
		\includegraphics[width=0.9 \textwidth]{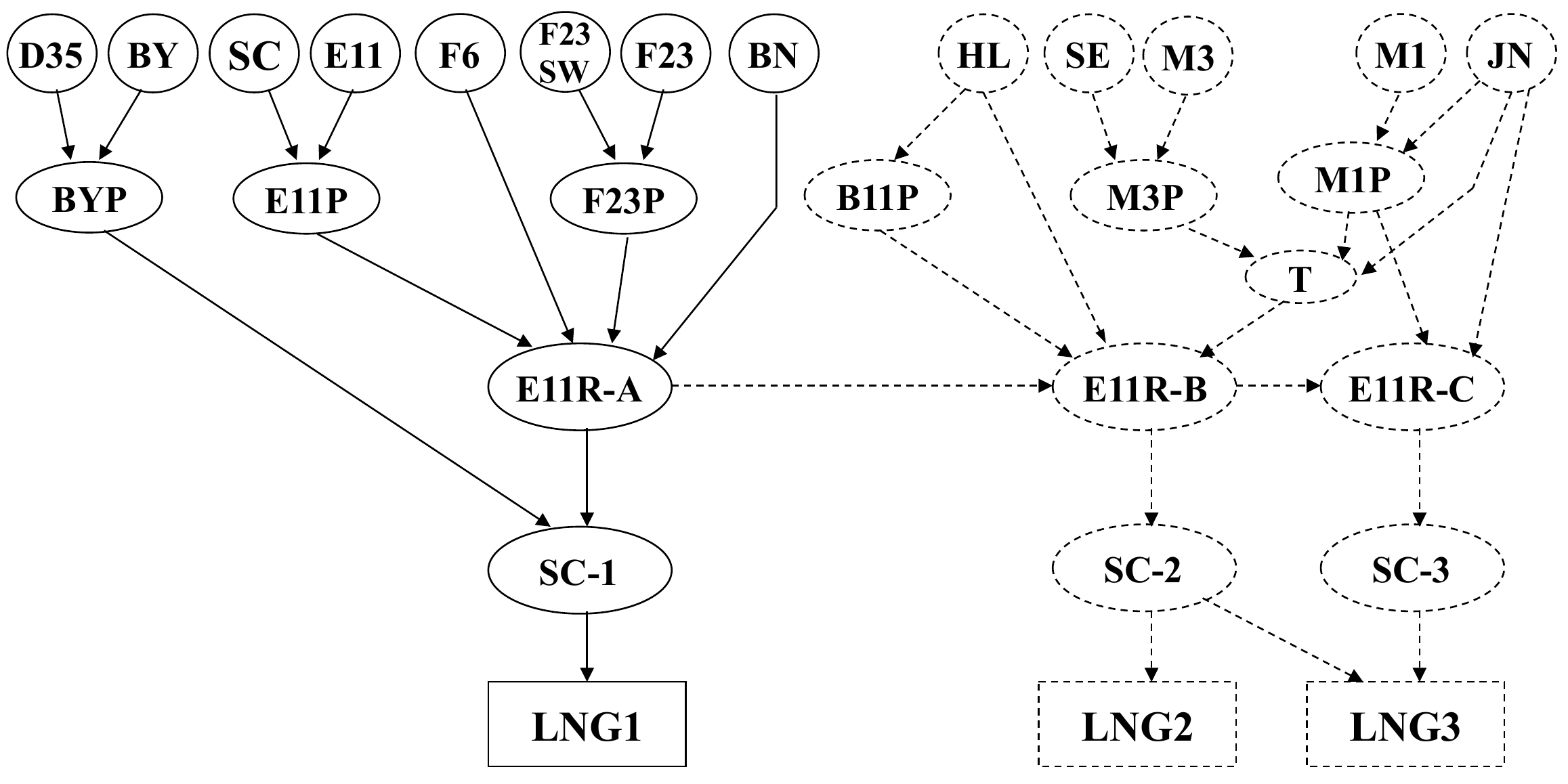}
		\caption{Superstructure of case study B problem}      
		\label{fig:SGPS}	
	\end{center}	
\end{figure}

\subsection{Solution approaches and implementation}
The case studies were run on a virtual machine allocated with a 3.2GHz CPU. The virtual machine ran Linux operating system (Ubuntu 16.04) with 6 GB of memory. Three solution approaches were compared in the case studies: Monolith, JD1, JD2. Monolith refers to solving the problem using an off-the-shelf, general-purpose global optimization solver, JD1 refers to the basic JD algorithm, and JD2 refers to the enhanced JD algorithm. The case study problems and the subproblems required in JD1 and JD2 were all modeled on GAMS 24.7.4 \cite{Brook1988}, but JD1 and JD2 algorithms were programmed on MATLAB 2014a \cite{MATLAB2014a}. Data exchange between MATLAB and GAMS was realized via GAMS GDXMRW facility \cite{gdxmrw2011}. 

The monolith approach solved the problems using three global optimization solvers, i.e.,  ANTIGONE 1.1 \cite{misener-floudas:ANTIGONE:2014}, BARON 16.8 \cite{tawarmalani2004goo}, SCIP 3.2 \cite{Achterberg2009}. ANTIGONE 1.1 and BARON 16.8 adopted CONOPT 3 \cite{Drud1994} as its NLP solver and  CPLEX 12.6 \cite{ibmiic} as its LP/MILP solver, and SCIP 3.2 used its default solvers for the subproblems. JD1 and JD2 solved the problems by using CPLEX 12.6 for the LP/MILP subproblems and SCIP 3.2 for the nonconvex NLP/MINLP subproblems.

In JD2, the construction of Problems \eqref{eq:ODR} and \eqref{eq:JRMPR} require the convex relaxation of nonconvex sets $X_0$ and $Y_\omega$. In the case studies, $X_0$ was a mixed integer set defined by linear constraints, and it was relaxed into a polyhedral set via continuous relaxation. $Y_\omega$ was a nonconvex continuous set defined with bilinear functions, and it was relaxed into a polyhedral set via standard McCormick relaxation \cite{mccormick1976cog}. The relative and absolute termination tolerances for Case Study A were set to $10^{-3}$, and for Case study B were set to $10^{-2}$. JD1 and JD2 started with all design decisions being 0. 

During the execution of JD1 and JD2, large computing overhead may be incurred due to frequent model generation in GAMS and data exchange between GAMS and MATLAB. So both "Total solver time" and "Total run time" were recorded for the simulation studies, which refer to the total time for the subproblem solvers to solve each individual subproblem and the wall time for the entire solution procedure, respectively. The computing overhead could have be significantly reduced if JD1 and JD2 had been implemented using general-purpose programming languages, such as C++. For the monolith approach, the computing overhead was much less, as seen from the results in the next subsection.

\subsection{Results and discussion} \label{sec:sec6}

Summary of the results for case study A is presented on Tables \ref{tab:Table-1}, \ref{tab:Table-2}, \ref{tab:Table-3}. Table \ref{tab:Table-1} shows the results for the monolith approach using the three global optimization solvers. It can be seen that ANTIGONE was the fastest among the three solvers, but its solution time increased quickly with the problem size. BARON could also solve small problem instances quickly, but it could not find the desired $10^{-3}$-optimal solution (i.e., a solution with a relative gap no larger than $0.1\%$) for larger problem instances within the one hour run time limit. SCIP was the slowest of the three solvers; but unlike BARON, it happened to find the $10^{-3}$-optimal solution within one hour for all problem instances (but could not verify the optimality for large problem instances). On the other hand, Tables \ref{tab:Table-2}, \ref{tab:Table-3} show that both JD1 and JD2 could solve all problem instances fairly quickly. JD1 was not as fast as ANTIGONE or BARON for small problem instances, but its solution time increased more slowly than that of ANTIGONE or BARON. This was primarily because the number of JD1 iterations did not vary much with the number of scenarios. The nonconvex relaxed master problem \eqref{eq:JRMP} was the major contributor to JD1 solution time, and sometimes it dominated the solution time (as in the 64 scenario case). In JD2 where the relaxation of \eqref{eq:JRMP} (i.e., \eqref{eq:JRMPR}) is solved, the number of \eqref{eq:JRMP} needed to be solved was significantly reduced, and each \eqref{eq:JRMP} was much easier to solve due to extensive domain reduction. The price for reducing the \eqref{eq:JRMP} solution time was the time spent on optimization based domain reduction ODR$^k$, but the resulting total solution time still decreased for most cases, so JD2 generally outperformed JD1 and it scaled with the number of scenarios in a more consistent way. Note that Tables \ref{tab:Table-2} and \ref{tab:Table-3} do not include the times to solve easy LP and MILP subproblems like Problem \eqref{eq:BPP}, \eqref{eq:BFP}, \eqref{eq:LS0} and \eqref{eq:JRMPR}, because those times were very small compared to the total solution time.

\begin{table} [tb!] \small 
	
	\caption {Results for case study A - Monolith (Unit for time: seconds)}
	\label{tab:Table-1}     
	\centering

	\begin{tabular} { p{3.0cm} c  c  c  c  c  c} 
		
		\hline
		Number of scenarios       & 25 & 36 & 49  & 64 & 81 & 100  \\
		
		\hline
		ANTIGONE 1.1 & & & & & & \\
		Objective val. (\$)  	& -532.1 & -530.6	& -531.2 & -531.5 & -531.1	&-531.1  	\\ 

		Relative gap          & $\leq$0.1\% & $\leq$0.1\%     & $\leq$0.1\% & $\leq$0.1\% & $\leq$0.1\%             & $\leq$0.1\% \\		

		Total solver time   	& 12	&30& 95 & 242&  548	& 1470\\     
		
		Total run time   	& 13	&35& 112 & 284&  645	& 1703 \\  
		
		\hline
		
		BARON 16.8		& & & & & & \\
		Objective val. (\$)  	& -532.1 & -530.6	& -233.17 & -397.7&  -163.2	& -427.8 \\
		
		Relative gap                & $\leq$0.1\% & $\leq$0.1\%         & 63.6\% & 25.5\% & 69.6\%               & 22.2\% \\
		
		Total solver time   	& 18	&30& --$\dag$  & --$\dag$&  --  $\dag$	& -- $\dag$ \\
		
		Total run time   	& 20	&37& --$\dag$ & --$\dag$&  -- $\dag$	& -- $\dag$ \\
		
		\hline
		
		SCIP 3.2	& & & & & & \\
		Objective val. (\$)  	& -532.1 & -530.6	& -531.2 & -531.5 & -531.1	&-531.1  	\\ 
		
		Relative gap         & $\leq$0.1\% & $\leq$0.1\%     & 0.58\% & 2\% & 3.7\%             & 0.13\% \\
		
		Total solver time   	& 134	&1226& -- $\ddag$& -- $\ddag$&  -- $\ddag$	& -- $\ddag$ \\
		
		Total run time   	& 163	&1470& --$\ddag$ & -- $\ddag$&  -- $\ddag$	& -- $\ddag$ \\
		\hline                                          			
	\end{tabular}
	\\
	
	\vspace{0.1cm}
	\raggedright  \footnotesize{$\dag$ Solver terminated after the one hour time limit,  without finding the optimal solution.}	\\	
	\raggedright  \footnotesize{$\ddag$ Solver obtained the optimal solution after the one hour time limit, but did not reduce the gap to the set tolerance ($10^{-3}$).}

	\vspace{0.5cm}
	
	\caption{Results for case study A - JD1 (Unit for time: seconds)}
	\label{tab:Table-2}     
	\centering
	\setlength{\extrarowheight}{3pt}
	
	\begin{tabular} { p{4cm}  c   c  c  c  c  c } 
		
			\hline
Number of scenarios     &25  & 36  & 49 & 64 & 81 & 100    \\

\hline
Optimal obj. (\$) & -532.1	& -530.6 &	-531.2 & -531.5 & -531.1 & -531.1	\\ 
Relative gap          & $\leq$0.1\% & $\leq$0.1\%     & $\leq$0.1\% & $\leq$0.1\% & $\leq$0.1\%             & $\leq$0.1\% \\		

Num. of iterations  & 8 &13 &10 &14 & 10 & 12 	  \\
Num. of JRMP$^{(l)}$ solved  & 4	& 5	& 5 & 7 & 5 & 6	 \\
Time for JRMP$^{(l)}$ &6  &8	&11 & 519  & 122  &202 \\

Time for LS$^{k}_\omega$    & 49 &128   &108  & 188  & 179   &262 \\
Time for PP$^k$       &7  &25   &18  & 124  & 45   &66 \\
Total solver time     &63 &168  &141 & 840 & 352  &540 \\
Total run time        &139 & 479 &318 & 1223 & 677 &1020 \\    
\hline

\end{tabular}

\vspace{0.5cm}

\caption{Results for case study A - JD2 (Unit for time: seconds)}
\label{tab:Table-3}     
\centering
\setlength{\extrarowheight}{3pt}

\begin{tabular} { p{4cm}  c  c  c  c  c  c} 
\hline
Number of scenarios     & 25 &36 & 49& 64 & 81 &  100   \\

\hline

Objective val. (\$) &-532.1   &-530.5   & -531.2	& -531.5 &	-530.7 & -530.7 	\\ 

Relative gap          & $\leq$0.1\% & $\leq$0.1\%     & $\leq$0.1\% & $\leq$0.1\% & $\leq$0.1\%             & $\leq$0.1\% \\		

Num. of iterations  &	10&	10&	10 & 8 & 10 & 10\\

Num. of JRMPR$^{(l)}$ solved  & 7	&5	&6 & 4 &4 & 5	\\

Num. of JRMP$^{(l)}$ solved & 3	&1	&3 & 1 & 1 & 2 \\
Time for JRMP$^{(l)}$ &1    &2	  &2     & 2   & 3    & 5 \\

Time for LS$^{k}_\omega$ &51	&71   &103    &110   & 165    & 190\\
Time for PP$^{k}$     &10    &22    &24    & 47    & 66    & 37 \\
Time for ODR$^k$       &35 	&44	  &55    & 61  & 104 & 140\\
Total solver time  	  &100 	&142 &192   & 283   & 345  & 391\\
Total run time 	   	  &210  &308  &406  & 549   & 739   & 968 \\                             
\hline
\end{tabular}

\end{table}

Tables \ref{tab:Table-4} and \ref{tab:Table-5} present the results for case study B. ANTIGONE actually found the desired $10^{-2}$-optimal solution, but it cannot reduce the gap to $1\%$ within the 24 hour run time limit; for the 25-scenario instance, it mistakenly terminated before the run time limit without reducing the gap to $1\%$. BARON had the similar problem; it obtained the $10^{-2}$-optimal solution for most problem instances but could not reduce the gap to $1\%$ for any problem instance. SCIP performed better than ANTIGONE and BARON for case study B, but it could only solve the 25 scenario and 36 scenario problem instances successfully. JD1 could not solve large problem instances either, because the \ref{eq:JRMP} subproblems took too much time to solve and the solution procedure could not terminate with the time limit. Table \ref{tab:Table-5} shows that JD2 solved the all problem instances successfully, and its solution time scaled well with the number of scenarios. This is because the total number of JD2 iterations did not vary significantly with the number of scenarios, and the times for \ref{eq:JRMP} and domain reduction did not increase greatly with the number of scenarios. It can be seen that for this problem, domain reduction, primarily \eqref{eq:ODR}, dominated the total solution time, so a more efficient way to perform domain reduction could have been able to effectively reduce the solution time. This case study problem indicates that, general-purpose global optimization solvers may not be able to effectively exploit the structure of a complex nonconvex MINLP and solve the problem efficiently enough, and this is when one might consider the use of a tailored decomposition strategy like the one proposed in this paper.

\begin{table} [tb!] \small 
	
	\caption {Results for case study B - Monolith (Unit for time: sec)}
	\label{tab:Table-4}     
	\centering

	\begin{tabular} { p{3.5cm}   c  c  c  c  c  c } 
		\hline
		Number of scenarios       &  25 & 36  & 49 & 64 & 81 & 100\\			
		\hline
		ANTIGONE 1.1 & & & & & & \\
		Objective val. (Billion \$) & -33.87& -33.67 &-33.81 & -33.76 & -33.78 & -33.79\\ 
		Relative gap.   	& 1.4\%& 2.1\% & 1.7\%& 1.8\% & 1.8\%& 1.7\% \\ 
		Total solver time   	& 51465$\dag$& -- $^\ddag$ & -- $\ddag$& --$\ddag$ & -- $\ddag$  & -- $\ddag$ \\
		Total run time   	&58522$\dag$ & --$\ddag$ & -- $\ddag$ &-- $\ddag$ & -- $\ddag$ & -- $\ddag$ \\
		
		\hline
		BARON 16.8 & & & & & & \\
		Objective val. (Billion \$) & -33.87& -33.91 &-33.90 & -33.31 & -33.91 & -33.79\\ 
		Relative gap.   	& 1.4\%& 1.3\% & 1.3\%& 3.6\% & 1.3\% & 1.6\% \\ 
		Total solver time   	&40530  $\dag$&59965 $\dag$ & 58460$\dag$ & -- $\ddag$ & --$\ddag$  & --$\ddag$  \\
		Total run time   	&68060  $\dag$&69520 $\dag$ & 70196 $\dag$ & -- $\ddag$ & -- $\ddag$ & -- $\ddag$ \\
		\hline
		SCIP 3.2  & & & & & & \\
		Objective val. (Billion \$) & -33.92& -33.91 &-33.81 & -33.76 & -33.78 & -33.77\\ 
		Relative gap. 	&$\leq$1\% & $\leq$1\% & 1.52\% & 1.69\% & 1.69\% & 1.75\% \\ 
		Total solver time 	&54337 &11952  & { --$\ddag$}  &{--$\ddag$}  & { --$\ddag$} & { --$\ddag$}  \\
		Total run time  	&61365 &13316 & { --$\ddag$}  &{--$\ddag$}  & { --$\ddag$}  & { --$\ddag$}  \\
		
		\hline                                              			
	\end{tabular}
	\vspace{0.1cm}
	
	\raggedright  \footnotesize{$\dag$ Solver terminated with a nonzero exit code within 24 hours, and the relative gap was larger than the set tolerance ($10^{-2}$).} \\
	\raggedright  \footnotesize{$^\ddag$ Solver terminated after the 24 hour time limit, with a relative gap larger than the set tolerance ($10^{-2}$).} \\	
	\vspace{0.5cm}
	
	\caption{Results for case study B - JD2 (Unit for time: sec)}
	\label{tab:Table-5}     
	\centering
	\setlength{\extrarowheight}{3pt}
	
	\begin{tabular} { p{4.2cm}  c   c  c  c  c  c  } 
		
		\hline
		Number of scenarios     & 25& 36 & 49 & 64 & 81 & 100  \\
		
		\hline
				
		Objective val. (Billion \$)&-33.58 & -33.57 &-33.77 &-33.71 &-33.57 &-33.55\\
		
		Relative gap          & $\leq$1\% & $\leq$1\%     & $\leq$1\% & $\leq$1\% & $\leq$1\%             & $\leq$1\% \\		
		
		Num. of iterations  &27 	&24    &30 &25 &23 &23   \\
		Num. of JRMPR$^{(l)}$ solved  &21  &17 &23 &17 &16 &14    \\
		Num. of JRMP$^{(l)}$ solved & 17   & 10 &15 &7 &8 &6 \\
		Time for JRMP$^{(l)}$ & 948 &696	&3547 &1617 &3948 &5651 \\

Time for LS$^k_\omega$    & 5676	&3820	& 14279 & 2734 & 2188 &2814   \\
		Time for PP$^k$   &155	&443 &560 &509 &388 &1000 \\

		Time for ODR$^k$ &7203	&9247	&19020 &22661 &21137 &30961 \\
		Total solver time  	&14028	&14288	& 37832 &27702 &27893 & 40769 \\
		Total run time 	   &16431	&16482	&44525 & 32150 &33271 & 47483 \\                             
		\hline
		
	\end{tabular}

\end{table}



\section{Concluding Remarks}
Two joint decomposition methods, JD1, and JD2, are developed in this paper for efficient global optimization of Problem \eqref{eq:P}. JD1 is a basic joint decomposition approach, which follows the notions of classical decomposition methods as well as convex relaxation, in order to solve \eqref{eq:P} via solving a sequence of relatively easy subproblems. JD2 is an enhanced version of JD1 that integrates several domain reduction techniques. It has been proved that both methods can terminate in a finite number of iterations with an $\epsilon$-optimal solution if some mild conditions are satisfied. 

We considered two case study problems that come from integrated design and operation under uncertainty, in order to demonstrate the potential computational advantages of joint decomposition. For the first problem which is smaller and easier, both JD1 and JD2 outperformed state-of-the-art global solvers when the number of scenarios was large, and JD2 generally outperformed JD1. For the second problem which was larger and more difficult, JD2 outperformed state-of-the-art global solvers and JD1 (which could not close the gap for most cases). The case study results indicate that, when joint decomposition can effectively exploit the problem structure, the total number of iterations it requires does not increase significantly with the number of scenarios, and consequently the solution time increases slowly with the problem size compared to the general-purpose global optimization solvers. On the other hand, like all decomposition methods, joint decomposition uses existing solvers to solve its subproblems, so its computational performance does rely on the advances in general-purpose local and global optimization solvers.

In this paper, we only consider domain reduction for the linking variables in $x_0$. In the future, we will also consider domain reduction for some key non-linking complicating variables in $y_\omega$ that influence the convergence rate the most, and investigate how to find out these key variables. This can effectively tighten the convex relaxation of Problem \eqref{eq:JRMP}, and therefore reduce the number of \ref{eq:JRMP} to be solved and accelerate the solution of each \ref{eq:JRMP}.

\section*{Acknowledgement}
The authors are grateful to the discovery grant (RGPIN 418411-13) and the collaborative research and development grant (CRDPJ 485798-15) from Natural Sciences and Engineering Research Council of Canada (NSERC).

\setcounter{equation}{0}
\renewcommand\theequation{A.\arabic{equation}}

\section*{Appendix A: Reformulation from \eqref{eq:P0} to \eqref{eq:P}}
 From Problem \eqref{eq:P0}, We first separate the convex part and the nonconvex part of the problem. Specifically, let  $v_{\omega}=(v_{c,\omega},v_{nc,\omega})$, where $v_{c,\omega}$ includes variables that are only involved in convex functions and restricted by convex constraints/sets, and $v_{nc,\omega}$ includes the variables that are involved in a nonconvex function and/or restricted by a nonconvex constraint/set. In addition, we introduce duplicate variables $v_{0,1},...,v_{0,s}$ for variable $x_0$, to express the relation among all scenarios using NACs. We then rewrite Problem \eqref{eq:P0} as: 
\begin{equation} \label{eq:P1} 
\begin{split} 
& \min_{\substack{x_0, v_{0,1},...,v_{0,s} \\ v_{c,1},...,v_{c,s} \\ v_{nc,1},...,v_{nc,s}}} \;\sum_{\omega=1}^s[f_{0,\omega}(v_{0,\omega}) +f_{c,\omega}(v_{c,\omega}) + f_{nc,\omega}(v_{nc,\omega})]\\ 
& \textrm{s.t.}  \;\;\;  x_{0}  = v_{0,\omega}, \quad \forall \omega \in \{1,...,s\}, \\   
& \quad \quad g_{0,\omega}(v_{0,\omega}) + g_{c,\omega}(v_{c,\omega}) + g_{nc,\omega}(v_{nc,\omega}) \le 0, \quad \forall \omega \in \{1,...,s\}, \\   
& \quad \quad  x_{0} \in {X_0}, \\
& \quad \quad v_{0,\omega} \in \hat{X}_{0}, \; v_{c,\omega} \in V_{c,\omega}, \; {v_{nc,\omega}} \in {V_{nc,\omega}}, \quad \forall \omega \in \{1,...,s\}. \\  
\end{split}
\end{equation}
In the above formulation, set $X_0 \subset \mathbb{R}^{n_0}$ is either convex or nonconvex, set $V_{c,\omega} \subset \mathbb{R}^{n_c}$ is convex, set $V_{nc,\omega} \subset \mathbb{R}^{n_{nc}}$ is either convex or nonconvex. Functions $f_{c,\omega}: V_{c,\omega} \rightarrow \mathbb{R}$ and $g_{c,\omega}: V_{c,\omega} \rightarrow \mathbb{R}^{m_{c}}$ are convex. Functions $f_{nc,\omega}: V_{nc,\omega} \rightarrow \mathbb{R}$, $g_{nc,\omega}: V_{nc,\omega} \rightarrow \mathbb{R}^{m_{nc}}$, $f_{0,\omega}$, and $g_{0,\omega}$ are either convex or nonconvex. Set $\hat{X}_0 \in \mathbb{R}^{n_0}$ is a convex relaxation of $X_0$ (and it is same to $X_0$ if $X_0$ is convex). The restriction $z_{0,\omega} \in \hat{X}_{0}$ is actually redundant with the presence of NACs; however, it tightens the problem when the NACs are dualized. Note that in order to generate a convex relaxation of $X_0$, extra variables may be introduced \cite{gatzke2002coc}, so the dimension of the relaxation may be larger than that of $X_0$. Here $\hat{X}_0$ can be understood as the projection of the relaxation set on the $\mathbb{R}^{n_0}$ space. For simplicity of notation, in this paper we always express a convex relaxation (of a set or a function) on the original variable space and do not explicitly show the extra variables needed for constructing the relaxation.

Define new variables $t_\omega$, $\alpha_{c,\omega}$, $\alpha_{nc,\omega}$, $\beta_{c,\omega}$, $\beta_{nc,\omega}$, such that Problem \eqref{eq:P1} can be written as:
\begin{equation} 
\begin{split} 
& \min \;\sum_{\omega=1}^s t_\omega \\ 
& \textrm{s.t.}  \;\;\;  x_{0}  = v_{0,\omega}, \quad \forall \omega \in \{1,...,s\}, \\   
& \quad \quad \beta_{c,\omega} + \beta_{nc,\omega} \le 0, \quad \forall \omega \in \{1,...,s\}, \\ 
& \qquad t_\omega \ge \alpha_{c,\omega}+\alpha_{nc,\omega}, \quad \forall \omega \in \{1,...,s\}, \\
& \qquad \alpha_{c,\omega} \ge f_{c,\omega}(v_{c,\omega}), \quad \forall \omega \in \{1,...,s\}, \\
& \qquad \alpha_{nc,\omega} \ge f_{0,\omega}(v_{0,\omega}) + f_{nc,\omega}(v_{nc,\omega}), \quad \forall \omega \in \{1,...,s\}, \\
& \qquad \beta_{c,\omega} \ge g_{0,\omega}(v_{0,\omega}) + g_{c,\omega}(v_{c,\omega}), \quad \forall \omega \in \{1,...,s\}, \\
& \qquad \beta_{nc,\omega} \ge g_{nc,\omega}(v_{nc,\omega}), \quad \forall \omega \in \{1,...,s\}, \\
& \quad \quad  x_{0} \in {X_0}, \\
& \quad \quad v_{0,\omega} \in \hat{X}_{0}, \;  v_{c,\omega} \in V_{c,\omega}, \; {v_{nc,\omega}} \in {V_{nc,\omega}}, \quad \forall \omega \in \{1,...,s\}. \\  
\end{split}
\end{equation}
Define $x_\omega=(v_{0,\omega}, v_{c,\omega}, t_\omega, \alpha_{c,\omega}, \beta_{c,\omega})$, $y_\omega=(v_{nc,\omega}, \alpha_{nc,\omega}, \beta_{nc,\omega})$, then the above formulation can be written as the following Problem \eqref{eq:P}:
\begin{equation} \tag{P}
\begin{split} 
& \min \;\sum_{\omega=1}^s c_\omega^Tx_{\omega}\\ 
& \textrm{s.t.}  \quad  \; \; x_{0}  = H_{\omega}x_{\omega}, \quad \forall \omega \in \{1,...,s\}, \\   
& \quad \quad \; \;\; A_{\omega}x_{\omega} + B_{\omega}y_{\omega} \le 0, \quad \forall \omega \in \{1,...,s\}, \\   
& \quad \quad \;\;\; x_{0} \in {X_0}, \\
& \quad \quad \;\;\; x_{\omega} \in X_{\omega}, \; {y_{\omega}} \in {Y_{\omega}}, \quad \forall \omega \in \{1,...,s\}, \\ 
\end{split}
\end{equation}
where the matrices 
\[c_\omega=\left[\begin{array}{c}
 0 \\ 0 \\ I \\ 0 \\ 0 \\
 \end{array} \right], \quad 
H_\omega=\left[I \;\; 0 \;\; 0 \;\; 0 \;\; 0 \right], \quad 
A_\omega=\left[
\begin{array} {ccccc}
0 & 0 & 0 & 0 & I \\ 0 & 0 & -I & I & 0 \\
\end{array}
\right], \quad 
B_\omega=\left[
\begin{array} {ccc}
0 & 0 & I \\ 0 & I & 0 \\
\end{array}
\right], \quad 
 \]
and the sets
\begin{align*}
X_\omega= & \{(v_{0,\omega}, v_{c,\omega}, t_\omega, \alpha_{c,\omega}, \beta_{c,\omega}) : v_{0,\omega} \in \hat{X}_{0}, \; v_{c,\omega} \in V_{c,\omega}, \\
& \quad \alpha_{c,\omega} \ge f_{c,\omega}(v_{c,\omega}), \;\; \beta_{c,\omega} \ge g_{0,\omega}(v_{0,\omega}) + g_{c,\omega}(v_{c,\omega})\}, \\
Y_\omega= & \{(v_{nc,\omega}, \alpha_{nc,\omega}, \beta_{nc,\omega}) : {v_{nc,\omega}} \in {V_{nc,\omega}}, \; \alpha_{nc,\omega} \ge f_{0,\omega}(v_{0,\omega}) +f_{nc,\omega}(v_{nc,\omega}), \\
& \quad \beta_{nc,\omega} \ge g_{nc,\omega}(v_{nc,\omega}) \}.
\end{align*}
The "0" and "I" in the matrices represent zero and identity matrices, and their dimensions are conformable to 
the relevant variables. According to the convexity/nonconvexity of the functions and the sets stated before, set $x-\omega$ is convex and set $y_\omega$ is nonconvex.

\setcounter{equation}{0}
\renewcommand\theequation{B.\arabic{equation}}

\section*{Appendix B: The stochastic pooling problem with mixed-integer first-stage decisions}
The two-stage stochastic pooling problem from Li et al. \cite{li2010spp} is modified here to address continuous design (first-stage) decisions. The nomenclature used in \cite{li2010spp} is adopted to describe the model, in which the scenarios are indexed by $h$ (rather than $\omega$). 

In the modified model, the design decisions on sources, pools, product terminals, denoted by $y_i^{\text{\tiny{S}}}$, $y_j^{\text{\tiny{P}}}$, $y_k^{\text{\tiny{T}}}$, can be continuous, integer, or mixed integer. If $y_i^S \in \{0,1\}$, then the design decision is to determine whether source $i$ is to be developed, and the related parameter $Z_{i}^{\text{\tiny{UB}}}$ represents the fixed capacity of the source. If $y_i^S$ is continuous and $y_i^{\text{\tiny{S}}} \in [0,1]$, then it is a capacity design decision, specifically it represents the ratio of source $i$ capacity to the maximum allowed capacity of the source (denoted by $Z_{i}^{\text{\tiny{UB}}}$). The design decisions on the pipelines among sources, pools, and terminals are all continuous, denoted by $y_{i,j}^{\text{\tiny{SP}}}$, $y_{i,k}^{\text{\tiny{ST}}}$, $y_{j,j^-}^{\text{\tiny{PP}}}$, $y_{j,k}^{\text{\tiny{PT}}} \in [0,1]$. They represents the ratios of the pipeline capacities to the maximum allowed capacities (denoted by $F^{\text{\tiny{SP,UB}}}_{i,j}$, $F^{\text{\tiny{ST,UB}}}_{i,k}$, $F^{\text{\tiny{PP,UB}}}_{j,j^-}$, $F^{\text{\tiny{PT,UB}}}_{j,k}$). 

All design and operational decision variables are nonnegative, and we do not impose other lower bounds on these variables in order to simplify discussion. The new stochastic pooling model consists primarily of three submodels, for the sources, pools, and product terminals, respectively. 

\subsection{Model for the sources}
The following group of constraints \eqref{eq:source} represents the submodel for the sources. Eq. (B.1a-B.1c) are same to Eq. (12-14) in \cite{li2010spp}, except that the lower flow bounds are not imposed. Eq. (B.1d-B.1f) are developed in place of the topology constraints Eq. (15-16) (which are invalid for continuous design decisions). Eq. (B.1d-B.1e) limit the capacity of a pipeline by the capacity of the source it connects. If $y_i^{\text{\tiny{S}}}=0$, then there cannot exist a pipeline connecting it, in other words, the capacity of a pipeline connecting it has to be zero. Eq. (B.1f) requires that the total capacity of all pipelines connecting to a source should be no less than the capacity of the source. This is to ensure enough pipeline capacity to move all materials generated in the source to other parts of the system in real-time.

 \begin{subequations} \label{eq:source}
 \begin{align}
 &  \sum_{j \in \Theta^{\text{\tiny{SP}}}_i} f^{\text{\tiny{SP}}}_{i,j,h} 
 + \sum_{k \in \Theta^{\text{\tiny{ST}}}_i} f^{\text{\tiny{ST}}}_{i,k,h} \leq y^{\text{\tiny{S}}}_i Z_{i}^{\text{\tiny{UB}}}, \\
& f^{\text{\tiny{SP}}}_{i,j,h}
\leq y^{\text{\tiny{SP}}}_{i,j} F^{\text{\tiny{SP,UB}}}_{i,j}, \\
& f^{\text{\tiny{ST}}}_{i,k,h}
\leq y^{\text{\tiny{ST}}}_{i,k} F^{\text{\tiny{ST,UB}}}_{i,k}, \\
& y_{i,j}^{\text{\tiny{SP}}} F_{i,j}^{\text{\tiny{SP,UB}}} \le y_i^{\text{\tiny{S}}} Z_{i}^{\text{\tiny{UB}}} ,\\
&  y_{i,k}^{\text{\tiny{ST}}} F_{i,k}^{\text{\tiny{ST,UB}}} \le y_i^{\text{\tiny{S}}} Z_{i}^{\text{\tiny{UB}}}, \\
& y_i^{\text{\tiny{S}}} Z_{i}^{\text{\tiny{UB}}}  \le \sum_{j \in \Theta_i^{\text{\tiny{SP}}}} y_{i,j}^{\text{\tiny{SP}}} F_{i,j}^{\text{\tiny{SP,UB}}} + \sum_{k \in \Theta_i^{\text{\tiny{ST}}}}
y_{i,k}^{\text{\tiny{ST}}} F_{i,k}^{\text{\tiny{ST,UB}}}, \\
& \forall i \in \{1,...,n\} , \; \forall j \in \Theta^{\text{\tiny{SP}}}_i, \; \forall k \in \Theta^{\text{\tiny{ST}}}_i , \; \forall h \in \{1,...,b\}. \nonumber 
 \end{align}
 \end{subequations}

\subsection{Model for the pools}
The following group of constraints \eqref{eq:pool} represents the submodel for the pools. Eq. (B.2a-B.1e) are same to Eq. (17-21) in \cite{li2010spp}, except that the lower flow bounds are not imposed. Eq. (B.2f-B.2k) are developed in place of the topology constraints (23-26) in \cite{li2010spp}. The interpretation of Eq. (B.2f-B.2k) is similar to that of Eq. (B.1d-B.1f) and therefore omitted. 

 \begin{subequations} \label{eq:pool}  
 \begin{align}
& f^{\text{\tiny{PT}}}_{j,k,w,h} = s^{\text{\tiny{PT}}}_{j,k,h} \left( \sum_{i \in \Omega^{\text{\tiny{SP}}}_j} f^{\text{\tiny{SP}}}_{i,j,h} U_{i,w,h} 
+ \sum_{j^+ \in \Omega^{\text{\tiny{PP+}}}_j} f^{\text{\tiny{PP}}}_{j^{+},j,w,h} \right), \\
& f^{\text{\tiny{PP}}}_{j,j^{-},w,h} = s^{\text{\tiny{PP}}}_{j,j^{-},h} \left( \sum_{i \in \Omega^{\text{\tiny{SP}}}_j} f^{\text{\tiny{SP}}}_{i,j,h} U_{i,w,h} 
+ \sum_{j^+ \in \Omega^{\text{\tiny{PP+}}}_j} f^{\text{\tiny{PP}}}_{j^{+},j,w,h} \right), \\
& \sum_{j^{-} \in \Omega^{\text{\tiny{PP--}}}_j} s^{\text{\tiny{PP}}}_{j,j^{-},h}
+ \sum_{k \in \Omega^{\text{\tiny{PT}}}_j} s^{\text{\tiny{PT}}}_{j,k,h} = 1, \quad s^{\text{\tiny{PP}}}_{j,j^-,h},s^{\text{\tiny{PT}}}_{j,k,h} \geq 0, \\
& y^{\text{\tiny{PP}}}_{j,j^-} F^{\text{\tiny{PP,LB}}}_{j,j^{-}} \leq \hspace{-5pt} \sum_{w \in \{1,...,l\}} \hspace{-5pt} f^{\text{\tiny{PP}}}_{j,j^-,w,h}
\leq y^{\text{\tiny{PP}}}_{j,j^-} F^{\text{\tiny{PP,UB}}}_{j,j^{-}}, \\
& y^{\text{\tiny{PT}}}_{j,k} F^{\text{\tiny{PT,LB}}}_{j,k} \leq \sum_{w \in \{1,...,l\}} f^{\text{\tiny{PT}}}_{j,k,w,h}
\leq y^{\text{\tiny{PT}}}_{j,k} F^{\text{\tiny{PT,UB}}}_{j,k}, \\
   	& y_j^{\text{\tiny{P}}} Z_{j}^{\text{\tiny{P,UB}}} \ge y_{i,j}^{\text{\tiny{SP}}} F_{i,j}^{\text{\tiny{SP,UB}}},  \\
   	& y_j^{\text{\tiny{P}}} Z_{j}^{\text{\tiny{P,UB}}} \ge y_{j^+,j}^{\text{\tiny{PP}}} F_{j^+,j}^{\text{\tiny{PP,UB}}}, \\ 
   	& y_j^{\text{\tiny{P}}} Z_{j}^{\text{\tiny{P,UB}}} \ge y_{j,j^-}^{\text{\tiny{PP}}} F_{j,j^-}^{\text{\tiny{PP,UB}}}, \\
   	& y_j^{\text{\tiny{P}}} Z_{j}^{\text{\tiny{P,UB}}} \ge y_{j,k}^{\text{\tiny{PT}}} F_{j,k}^{\text{\tiny{PT,UB}}}, \\
   	& y_j^{\text{\tiny{P}}}  Z_{j}^{\text{\tiny{P,UB}}} \le \sum_{j^+ \in \Omega_j^{\text{\tiny{PP+}}}} y_{j^+,j}^{\text{\tiny{PP}}} F_{j^+,j}^{\text{\tiny{PP,UB}}} + \sum_{i \in \Omega_j^{\text{\tiny{SP}}}} y_{i,j}^{\text{\tiny{SP}}} F_{i,j}^{\text{\tiny{SP,UB}}}, \\
   	& y_j^{\text{\tiny{P}}} Z_{j}^{\text{\tiny{P,UB}}} \le \sum_{j^- \in \Omega_j^{\text{\tiny{PP-}}} } y_{j,j^-}^{\text{\tiny{PP}}} F_{j,j^-}^{\text{\tiny{PP,UB}}} + \sum_{k \in \Omega_j^{\text{\tiny{PT}}}} y_{j,k}^{\text{\tiny{PT}}} F_{j,k}^{\text{\tiny{PT,UB}}}, \\
& \forall j \in \{1,...,r\}, \; \forall j^{-} \in \Omega^{\text{\tiny{PP-}}}_j, \; \forall k \in \Omega^{\text{\tiny{PT}}}_j, \; \forall w \in \{1,...,l\}, \; \forall h \in \{1,...b\}. \nonumber    	
\end{align}
\end{subequations}

 \subsection{Model for the product terminals}
The following group of constraints \eqref{eq:term} represents the submodel for the terminals. Eq. (B.3a-B.3b) are same to Eq. (27-28) in \cite{li2010spp}, except that the lower flow bounds and content bounds are not imposed. Again, Eq. (B.3c-B.3e) are developed in place of the old topology constraints that are invalid for continuous design decisions (i.e., Eq. (23-26) in \cite{li2010spp}).

  \begin{subequations} \label{eq:term}
  \begin{align}
  & \sum_{j \in \Pi^{\text{\tiny{PT}}}_{k}}\sum_{w \in \{1,...,l\}} f^{\text{\tiny{PT}}}_{j,k,w,h}
    + \sum_{i \in \Pi^{\text{\tiny{ST}}}_{k}} f^{\text{\tiny{ST}}}_{i,k,h} \leq y^{\text{\tiny{T}}}_{k} D_{k,h}^{\text{\tiny{UB}}},  
   \\
  & \sum_{j \in \Pi^{\text{\tiny{PT}}}_{k}} f^{\text{\tiny{PT}}}_{j,k,w,h}
      + \sum_{i \in \Pi^{\text{\tiny{ST}}}_{k}} f^{\text{\tiny{ST}}}_{i,k,h} U_{i,w,h} \leq \nonumber \\
  & \left( \sum_{j \in \Pi^{\text{\tiny{PT}}}_{k}}\sum_{w \in \{1,...,l\}} f^{\text{\tiny{PT}}}_{j,k,w,h}
        + \sum_{i \in \Pi^{\text{\tiny{ST}}}_{k}} f^{\text{\tiny{ST}}}_{i,k,h} \right) V_{k,w}^{\text{\tiny{UB}}} \\
       & y_k^{\text{\tiny{T}}} D_{k}^{\text{\tiny{UB}}} \ge y_{i,k}^{\text{\tiny{ST}}} F_{i,k}^{\text{\tiny{ST,UB}}} \\
       &  y_k^{\text{\tiny{T}}} D_{k}^{\text{\tiny{UB}}} \ge y_{j,k}^{\text{\tiny{PT}}} F_{j,k}^{\text{\tiny{PT,UB}}}, \\
    & y_k^{\text{\tiny{T}}} D_{k}^{\text{\tiny{UB}}} \le \sum_{i \in \Pi_{k}^{\text{\tiny{ST}}}} y_{i,k}^{\text{\tiny{ST}}} F_{i,k}^{\text{\tiny{ST,UB}}} +\sum_{k \in \Pi_{k}^{\text{\tiny{PT}}}} y_{j,k}^{\text{\tiny{PT}}} F_{j,k}^{\text{\tiny{PT,UB}}} \\              
 &   \forall k \in \{1,...,m\}, \; \forall w \in \{1,...,l\}, \; \forall h \in \{1,...,b\}.   \nonumber   
  \end{align}
  \end{subequations}
 
The modified stochastic pooling model can be stated as:
\begin{align*}
\text{minimize} & \quad \text{objective} \\
s.t. \quad    & \text{Eq. (B.1a-B.1f), Eq. (B.2a-B.2k), Eq. (B.3a-B.3e)}, \\
         & y_i^{\text{\tiny{S}}}, y_{j}^{\text{\tiny{P}}}, y_k^{\text{\tiny{T}}} \in \{0,1\} \;\; \text{or} \;\; [0,1], \\        
         & y_{i,j}^{\text{\tiny{SP}}}, y_{i,k}^{\text{\tiny{ST}}}, y_{j,j^-}^{\text{\tiny{PP}}}, y_{j,k}^{\text{\tiny{PT}}} \in [0,1], \\        
        & \text{all flow rates are nonnegative,} \\
        & \text{redudant constraints for accelerating global optimizaiton (Eq. (38-39) in \cite{li2010spp})}. \\
\end{align*}
 
The objective can be negative net present value, or negative annualized profit, as specified in \cite{li2010spp}.

\setcounter{equation}{0}
\renewcommand\theequation{C.\arabic{equation}}

\bibliographystyle{elsarticle-num}
\clearpage
\bibliography{references}

\newpage

\end{document}